\documentclass[reqno,11pt]{amsart}

\usepackage{amsmath,amsfonts,amssymb,graphicx,amsthm,enumerate,url}
\usepackage[noadjust]{cite}
\usepackage{stmaryrd}
\usepackage{mathrsfs,booktabs}
\usepackage{xifthen,xcolor,tikz,setspace}
\usetikzlibrary{decorations.pathmorphing,patterns,shapes,calc,snakes}
\usepackage[noadjust]{cite}
\usepackage[final]{showkeys} 

\definecolor{refkey}{gray}{.75}
\definecolor{labelkey}{gray}{.5}

\usepackage[colorlinks=true]{hyperref}
\colorlet{DarkGreen}{green!50!black}
\colorlet{DarkGray}{gray!60!black}

\numberwithin{equation}{section}

\renewcommand{\restriction}{\mathord{\upharpoonright}}
\renewcommand{\epsilon}{\varepsilon}
\newcommand{\given}{\;\big|\;}
\newcommand{\one}{\mathbf{1}}

\usepackage{color}
 \definecolor{refkey}{gray}{.5}
 \definecolor{labelkey}{gray}{.5}
\definecolor{light}{gray}{.9}

\newtheorem{maintheorem}{Theorem}

\newtheorem{theorem}{Theorem}[section]
\newtheorem*{theorem*}{Theorem}
\newtheorem{lemma}[theorem]{Lemma}
\newtheorem{claim}[theorem]{Claim}
\newtheorem{proposition}[theorem]{Proposition}

\newtheorem{fact}[theorem]{Fact}

\theoremstyle{definition}{

\newtheorem{definition}[theorem]{Definition}
\newtheorem*{definition*}{Definition}

\newtheorem{remark}[theorem]{Remark}
}

\renewcommand{\P}{\mathbb P}

\newcommand{\Z}{\mathbb Z}

\newcommand{\cC}{\ensuremath{\mathcal C}}

\newcommand{\cH}{\ensuremath{\mathcal H}}
\newcommand{\cI}{\ensuremath{\mathcal I}}

\newcommand{\cS}{\ensuremath{\mathcal S}}

\newcommand{\cZ}{\ensuremath{\mathcal Z}}
\newcommand{\llb }{\llbracket}
\newcommand{\rrb }{\rrbracket}

\renewcommand{\epsilon}{\varepsilon}

\newcommand{\tmix}{t_{\textsc{mix}}}

\newcommand{\gap}{\text{\tt{gap}}}
\newcommand{\tv}{{\textsc{tv}}}
\newcommand{\north}{{\textsc{n}}}
\newcommand{\south}{{\textsc{s}}}
\newcommand{\east}{{\textsc{e}}}
\newcommand{\west}{{\textsc{w}}}
\newcommand{\sd}{{\operatorname{sd}}}
\newcommand{\potts}{{\textsc{p}}}
\newcommand{\rc}{{\textsc{fk}}}

\newcommand{\alt}{{\mathrm{mixed}}}

\makeatletter
\newcommand{\superimpose}[2]{%
  {\ooalign{$#1\@firstoftwo#2$\cr\hfil$#1\@secondoftwo#2$\hfil\cr}}}
\makeatother

\begin{document}

\title[Boundary conditions and mixing at discontinuous phase transitions]{The effect of boundary conditions on mixing of \\ 2D Potts models at  discontinuous phase transitions}

\author{Reza Gheissari}
\address{R.\ Gheissari\hfill\break
Courant Institute\\ New York University\\
251 Mercer Street\\ New York, NY 10012, USA.}
\email{reza@cims.nyu.edu}

\author{Eyal Lubetzky}
\address{E.\ Lubetzky\hfill\break
Courant Institute\\ New York University\\
251 Mercer Street\\ New York, NY 10012, USA.}
\email{eyal@courant.nyu.edu}

\begin{abstract}
We study Swendsen--Wang dynamics for the critical $q$-state Potts model on the square lattice. For $q=2,3,4$, where the phase transition is continuous, the mixing time $\tmix$ is expected to obey a universal power-law independent of the boundary conditions. On the other hand, for large $q$, where the phase transition is discontinuous, the authors recently showed that $\tmix$ is highly sensitive to boundary conditions:  $\tmix \geq \exp(cn)$ on an $n\times n$ box with periodic boundary,  yet under free or monochromatic boundary conditions, $\tmix \leq\exp(n^{o(1)})$. 

In this work we classify this effect under boundary conditions that interpolate between these two (torus vs.\ free/monochromatic). 
Specifically, if one of the $q$ colors is red,
mixed boundary conditions such as red-free-red-free on the 4 sides of the box  induce $\tmix \geq \exp(cn)$, yet Dobrushin boundary conditions such as red-red-free-free, as well as red-periodic-red-periodic, induce sub-exponential mixing.
\end{abstract}

{\mbox{}
\vspace{-1.25cm}
\maketitle
}
\vspace{-0.5cm}

\section{Introduction}
The $q$-state Potts model at inverse temperature $\beta>0$ is a generalization of the Ising model ($q=2$) to $q\geq 3$ possible states. It is a canonical model of statistical physics and is one of the simplest models exhibiting a discontinuous (first-order) phase transition for some choices of $q$. Concretely, the model on a graph $G$ is a probability distribution over $\{1,...,q\}^{V(G)}$ with $\mu(\sigma)\propto \exp(\beta \sum_{ij\in E(G)} \boldsymbol 1\{\sigma_i=\sigma_j\})$.
Much of the analysis of the Potts model relies heavily on the random cluster (FK) model; the FK model is a model of dependent bond percolation parametrized by $(p,q)$, identified with the $q$-state Potts model via the Edwards--Sokal coupling \cite{EdSo88} when $q$ is integer and $p=1-e^{-\beta}$. 

On $\mathbb Z^2$, substantial recent progress has been made in understanding the Potts and random cluster phase transitions in $\beta$ and $p$, respectively.
On that geometry, the critical $p_c(q) = 1-e^{-\beta_c}$ was identified \cite{BeDu12} for all $q\geq 1$ with the self-dual point $p_\sd(q)= \frac {\sqrt q}{1+\sqrt q}$. 
It was shown in~\cite{DST15} that for $q\leq 4$, the phase transition is continuous (there is a unique infinite-volume Gibbs measure at $\beta_c(q)$) whereas for $q>4$, the phase transition is discontinuous ~\cite{DGHMT16} (there are $q+1$ extremal infinite-volume Gibbs measures corresponding to $q$ ordered phases and an additional disordered phase). In the latter case, the phase asymmetry at the critical point is expected to induce order-order and order-disorder surface tensions (known rigorously for $q$ large~\cite{LMMRS91,MMRS91}). In the present work we study the relationship between the order-disorder surface tension and the effect of boundary conditions on mixing times (time to reach equilibrium) for the critical 2D Potts model.

Specifically, we study the Swendsen--Wang dynamics~\cite{SW87}, a non-local Markov chain suggested in the physics literature as a fast MCMC sampler of the Potts model, as it switches between the different ordered phases (low-temperature bottlenecks) by moving through the FK representation of the Potts model using the Edwards--Sokal coupling.
The authors of~\cite{GL16a} analyzed the mixing times of the Swendsen--Wang dynamics at $\beta_c(q)$ on subsets of $\mathbb Z^2$ as the parameter $q$ varied. In \cite{GL16a}, polynomial and quasipolynomial upper bounds independent of the boundary conditions were proved for $q\leq 4$. When $q>4$, it was shown that $\tmix \geq \exp(cn)$ on the $n\times n$ torus (where the dynamics is slowed by the free energy barrier between the disordered phase and the $q$ ordered phases), while on the $n\times n$ box with free or red boundary conditions and large $q$, the authors proved that $\tmix \leq \exp(n^{o(1)})$.

\begin{figure}
\begin{tabular}{lcll}
\toprule%
Boundary Conditions & & & Swendsen--Wang \\
\midrule

Periodic/Mixed 
&  
\raisebox{-0.35in}{
\begin{tikzpicture}[scale=.20]
     \draw[draw=black, thick] (13,19) rectangle (20,26);
     \draw[draw=black, dashed, thick] (22,19) rectangle (29,26);
     \draw[draw=red,very thick] (22,19)--(29,19);
     \draw[draw=red,very thick] (22,26)--(29,26);
     
     \draw[draw=black, dashed, thick, fill=red, opacity=.05] (22,19) rectangle (29,26);
     
     \node[font=\tiny] at (16.5,26) {$||$};
     \node[font=\tiny] at (16.5,19) {$||$};
     \draw[draw=black] (12.4,22.5)--(13.6,22.5);
     \draw[draw=black] (19.4,22.5)--(20.6,22.5);
     
               \draw[ color=red, very thin, fill=white] (22,19) arc (-18:18:11.2) ;
     \draw[ color=red, very thin, fill=white] (29,19) arc (198:162:11.2) ;

  \end{tikzpicture}}
&
& $\tmix\geq \exp(cn)$ \\
\midrule[0.25pt]
\noalign{\medskip}
Dobrushin &  
\raisebox{-0.25in}{
\begin{tikzpicture}[scale=.2]
     \draw[draw=red, very thick, fill=red, opacity=.05] (13,10) rectangle (20,17);

     \draw[draw=red, very thick, fill=red, opacity=.05] (22,10) rectangle (29,17);
     \draw[fill=red, opacity=.05] (31,17) -- (38,17)--(38,10);

     \draw[draw=red, very thick] (13,10) rectangle (20,17);
     \draw[draw=black, dashed, thick] (22,10) rectangle (29,17);
     \draw[draw=black, dashed] (31,10) rectangle (38,17);

     \draw[draw=red,very thick] (22,10)--(22,17);
     \draw[draw=red,very thick] (22,17)--(29,17);
     \draw[draw=red,very thick] (29,10)--(29,17);
     \draw[draw=red,very thick] (31,17)--(38,17);
     \draw[draw=red,very thick] (38,10)--(38,17);
     
          \draw[ color=red, fill=red, opacity=.05] (31,17) arc (207:243:8) ;
     \draw[ color=red, very thin] (31,17) arc (207:243:8) ;

     \draw[ color=red, very thin, fill=white] (38,10) arc (27:63:8) ;

     \draw[ color=red, very thin, fill=white] (29,10) arc (72:108:11.2) ;

  \end{tikzpicture}}
 & & $\tmix\leq \exp(O(\sqrt{n\log n}))$ \\
\noalign{\medskip}
\midrule[0.25pt]
\noalign{\medskip}
Cylindrical 
&  
\raisebox{-0.25in}{
\begin{tikzpicture}[scale=.20]
     \draw[draw=red, very thick, fill=red, opacity=.05] (13,1) rectangle (20,8);
     \draw[draw=white, fill=red, opacity=.05] (22,1) rectangle (25,8);

     \draw[draw=black, thick] (13,1) rectangle (20,8);
     \draw[draw=black, dashed] (22,1) rectangle (29,8);
     \draw[draw=black,dashed] (31,1) rectangle (38,8);
     \draw[draw=red, very thick] (13,1)--(13,8);
     \draw[draw=red, very thick] (20,1)--(20,8);
     \draw[draw=black,thick] (22,1)--(29,1);
     \draw[draw=black,thick] (22,8)--(29,8);
     \draw[draw=black,thick] (31,1)--(38,1);
     \draw[draw=black,thick] (31,8)--(38,8);
     \draw[draw=red, very thick] (22,1)--(22,8);

     \node[font=\tiny] at (16.5,8) {$||$};
     \node[font=\tiny] at (16.5,1) {$||$};
     \node[font=\tiny] at (25.5,8) {$||$};
     \node[font=\tiny] at (25.5,1) {$||$};
     \node[font=\tiny] at (34.5,8) {$||$};
     \node[font=\tiny] at (34.5,1) {$||$};
     
          \draw[ color=red, very thin, fill=red, opacity=.05] (25,1) arc (-18:18:5.7) ;
     \draw[ color=red, very thin, fill=white] (25,8) arc (162:198:5.7) ;
     \draw[ color=red, very thin] (25,1) arc (-18:18:5.7) ;

  \end{tikzpicture}}
& &
 $\tmix\leq \exp(n^{1/2+o(1)})$ \\
\midrule[.25pt]
\bottomrule

\end{tabular}
 \caption{Mixing time bounds for Swendsen--Wang dynamics on $n\times n$ boxes with different sets of boundary conditions.   Dashed lines indicate free boundary conditions, the bold red lines denote red boundary conditions, and hash markings $|,||$ indicate periodic boundary conditions on their respective sides. The torus and all-red boundary conditions were considered in~\cite{GL16a}; the other examples are special cases of Theorems~\ref{mainthm:1}--\ref{mainthm:3}.
  }\label{fig:main-results}
 
\end{figure}

This sensitivity to boundary conditions is in analogy to the sensitivity of mixing to boundary conditions in the low-temperature Ising Glauber dynamics, where for low enough temperatures, the plus-minus surface tension is very well understood \cite{DKS} via cluster expansion. There, in $\mathbb Z^2$, the first sub-exponential bound of $\tmix \leq \exp(n^{\frac 12+o(1)})$ was obtained in~\cite{Martinelli94} for all sufficiently low temperatures under plus boundary. This was improved to $\tmix\leq \exp(n^{o(1)})$ in~\cite{MaTo10} and subsequently to $\tmix\leq n^{O(\log n)}$ for all $\beta>\beta_c$ in~\cite{LMST13}; it is believed to be of order $n^2$, governed by motion by mean-curvature~(cf.~\cite{FisherHuse87}).  However, at $\beta_c$, all known bounds are independent of the boundary conditions; in fact, it is believed that $\tmix \asymp n^{z}$ for some universal constant $z$. Such behavior should hold through $q\leq 4$. (See, e.g.,~\cite{GL16a,GL16b} for a more extensive account of related literature.)

When $q>4$ is sufficiently large, similar cluster expansion techniques (large $\beta$ is now replaced by large $q$) lead to an emergent order-disorder surface tension at the critical point, destroying the independence of mixing times and boundary conditions. Certain boundary conditions can destabilize the order-disorder phase symmetry, eliminating the exponential bottlenecks in the state space.
Thus, when the boundary conditions are monochromatic or free, the mixing time of Swendsen--Wang was shown~\cite{GL16a} to be~$\exp(n^{o(1)})$ and is believed to be $n^{O(1)}$, matching the above picture for the low temperature Ising model under plus boundary conditions.

In the present paper, we investigate the relationship between mixing times and boundary conditions that interpolate between periodic and free/monochromatic, at the critical point of a discontinuous phase transition (see Figure~\ref{fig:main-results}).
Our results hold for $q$ large enough, and are expected to hold whenever the order-disorder surface tension (see Def.~\ref{def:surface-tension}) is positive, which in turn should hold for all $q>4$.

For two sequences $f_n,g_n$, here and throughout the paper, we say $f_n \lesssim g_n$ if there exists $C>0$ such that $f_n \leq Cg_n$ for all $n$ and analogously define $f_n \gtrsim g_n$. We let $\Lambda_{n,m} = ([0,n]\times[0,m])\cap \mathbb Z^2$ with nearest-neighbor edges. The boundary of $\Lambda_{n,m}$, denoted $\partial \Lambda_{n,m}$, is the set of vertices of $\Lambda_{n,m}$ adjacent to $\mathbb Z^2 \setminus \Lambda_{n,m}$; let $\partial_{\north}\Lambda=\partial\Lambda_{n,m}\cap (\Z\times\{m\})$ be the its northern boundary of $\Lambda$, and define $\partial_{\east}\Lambda, \partial_\west \Lambda, \partial_\south \Lambda$ similarly.

\begin{maintheorem}[Mixed b.c.]\label{mainthm:1}
Let $q$ be large, $\epsilon>0$, and let $(a_n,b_n,c_n,d_n)$ be marked vertices on $\partial \Lambda_{n,n}$ such that they are not all within $\epsilon n$ of any one side of $\partial \Lambda_{n,n}$ and are all distance greater than $\epsilon n$ from each other. There exists $c(\epsilon,q)>0$ (independent of $n$ and $a_n,b_n,c_n,d_n$) such that the Swendsen--Wang dynamics on $\Lambda_{n,n}$ at $\beta=\beta_c(q)$ with boundary conditions that are red on the boundary segments $(a_n,b_n)$ and $(c_n,d_n)$ and free elsewhere, has
\[\tmix \gtrsim \exp(c n)\,.
\]
In particular, this holds with red boundary conditions on $\partial_{\east,\west} \Lambda_{n,n}$ and free elsewhere.
\end{maintheorem}

\begin{maintheorem}[Dobrushin b.c.] \label{mainthm:2}
Let $q$ be large and let $(a_n,b_n)$ be marked vertices on~$\partial \Lambda_{n,n}$. There exists $c(q)>0$ (independent of $n$ and $a_n,b_n$) so that  Swendsen--Wang dynamics at $\beta=\beta_c(q)$ with  boundary conditions that are red on the boundary segment~$(a_n,b_n)$ and free elsewhere, has
\[\tmix \lesssim \exp(c\sqrt{n\log n})\,.
\]
In particular, this holds with red boundary conditions on $\partial_{\south,\west} \Lambda_{n,n}$ and free elsewhere.
\end{maintheorem}

\begin{maintheorem}[Cylinders] \label{mainthm:3}
Let $q$ be large. The critical Swendsen--Wang dynamics with periodic boundary conditions on $\partial_{\north,\south} \Lambda_{n,n}$ and either red or free boundary conditions on each of $\partial_\east\Lambda_{n,n}$ and $\partial_\west \Lambda_{n,n}$ satisfies
\[\tmix\lesssim \exp(n^{1/2+o(1)})\,.
\]
\end{maintheorem}

\begin{remark}\label{rem:other-dynamics}
Theorems~\ref{mainthm:1}--\ref{mainthm:3} also hold for the Glauber dynamics for the random cluster (FK) model for $q$ sufficiently large (not necessarily integer). In fact, the proofs proceed by proving the desired result for the FK Glauber dynamics then using a priori estimates comparing $\tmix$ for the FK Glauber dynamics to that of Swendsen--Wang dynamics (cf.~\cite{Ul13,Ul14}) for the corresponding Potts model.  
\end{remark}

\begin{remark}
Theorems~\ref{mainthm:1}--\ref{mainthm:3} showed the dependence of mixing times on boundary conditions for the critical Potts model in the phase coexistence regime. If, instead, one were  interested in the simpler setting of  Glauber dynamics for the 2D Ising model at large  $\beta$ (where the plus and minus phases would assume the role of red and free phases in our theorems), the proofs would carry over and even simplify, via the tools of~\cite{DKS}. 
\end{remark}

\section{Preliminaries}\label{sec:preliminaries}

\subsection{The $q$-state Potts model}\label{sub:potts} In this section we formally introduce relevant facts about the Potts and random cluster models (for further details, see, e.g.,~\cite{Gr04}).

\subsubsection*{The Potts and FK models}
Define the $q$-state Potts model on a graph $G=(V,E)$ as the probability measure $\mu_{G,\beta,q}$ on $\Omega_{\textsc P}=[q]^V:= \{1,...,q\}^V$ as
\[\mu_{G,\beta,q} (\sigma)=\mathcal Z_{\textsc P}^{-1} e^{ \beta \sum_{(i,j)\in E} \boldsymbol 1\{\sigma_i=\sigma_j\}}\,,
\]
where the normalizer $\mathcal Z_{\textsc P}^{-1}$ is the \emph{partition function}.
Define the random cluster (FK) model on a graph $G$ as the probability measure $\pi_{G,p,q}$ on state space $\Omega_{\rc}=\{0,1\}^E$ as
\[\pi_{G,p,q}(\omega) = \mathcal Z^{-1}_{\rc} p^{o(\omega)}(1-p)^{|E|-o(\omega)} q^{k(\omega)}\,,
\]
where $o(\omega)=\sum_{e\in E} \omega(e)$ and $k(\omega)$ is the number of connected components (clusters) in the subgraph of $G$ induced by $\omega$ (we count singletons as their own clusters). We call edges that have $\omega(e)=1$ \emph{open}, or \emph{wired}, and edges that have $\omega(e)=0$, \emph{closed}, or \emph{free}.

\subsubsection*{Potts and FK boundary conditions}
Consider the Potts and FK models on $G=(V,E)$ with boundary $\partial G\subset V$. A Potts boundary condition on $\partial G\subset G$ is an assignment of spin values $\eta \in [q]^{\partial G}$ so that $\mu^\eta_{\beta,q,G}=\mu_{\beta,q,G}(\cdot \mid \sigma\restriction_{\partial G}=\eta)$. 

For a subset $\partial G\subset V$, and FK boundary condition on $\partial G$ is defined as follows: first augment $G$ to a graph $G'= (V, E')$ where $E'$ adds edges between any vertices in $\partial G$ not adjacent in $G$ and let $E'(\partial G)$ be the set of all edges between vertices in $\partial G$; then an assignment $\xi\in \{0,1\}^{E'}$ is an FK boundary condition on $\partial G$. Then $\xi$ can be identified with a partition of the vertices of $\partial G$, where the partition is given by the connected components of $\xi$. The FK measure with boundary conditions $\xi$ is denoted by $\pi^{\xi}_{p,q,G}$ and is given by counting $k(\omega)$ with connections from $\xi$ in mind. 

The red Potts boundary condition is an assignment of $\sigma(i)=1$ to all vertices of $\partial G$ where we always call the first state $\sigma(i)=1$ red, or $R$. The wired FK boundary condition is that in which all of $V(\partial G)$ is in the same boundary component. The free Potts boundary condition corresponds to $\partial G = \emptyset$ while the free FK boundary condition corresponds to the partition of $\partial G$ consisting only of singletons $\{\{v\}:v\in V(\partial G)\}$. The wired/red boundary conditions are \emph{ordered} as they pick out one of the $q$ ordered phases of the Potts model, whereas the free boundary conditions are \emph{disordered} as they pick out the disordered phase, where the $q$ states are symmetric.  

\subsubsection*{Edwards--Sokal coupling}
The Edwards--Sokal coupling \cite{EdSo88} is a coupling of the Potts and FK measures on a graph $G$ that enables us to reduce the study of the $q$-state Potts model, to the study of the FK model at integer $q$.
The joint probability assigned to $(\sigma,\omega)$, where $\sigma\in\Omega_\potts$ is a $q$-state Potts configuration at inverse-temperature $\beta>0$ and $\omega\in\Omega_\rc$ is an FK configuration with parameters $(p=1-e^{-\beta},q)$, is proportional to
\[\prod_{xy\in E(G)}\Big[(1-p)\one\{\omega(xy)=0\} + p\one\{\omega(xy)=1,\sigma(x)=\sigma(y)\}\Big]\,.\]

\subsubsection*{Planar duality}Throughout this paper we are concerned only with the Potts and FK models on planar graphs, and specifically rectangular subsets $\Lambda_{n,m}\subset \mathbb Z^2$ with vertices
\[\Lambda_{n,m}=\llb 0,n\rrb \times \llb 0,m\rrb := \{k\in \mathbb Z: 0\leq k \leq n\}\times \{k\in \mathbb Z: 0\leq k \leq m\}
\]
and nearest-neighbor edges. For a general subgraph $G\subset \mathbb Z^2$, the boundary $\partial G$ will be the set of vertices in $G$ with neighbors in $\mathbb Z^2-G$.  When considering rectangles and other graphs where it makes sense, the southern (bottom) boundary of $\Lambda$ is denoted $\partial_\south \Lambda_{n,m}= \llb 0,n\rrb \times \{0\}$ and $\partial_{\east,\north,\west}$ are analogously defined. Then for rectangles, we observe that $\partial \Lambda_{n,m} = \bigcup_{i\in \{\north,\south,\east,\west\}}\partial_{i}\Lambda_{n,m}$.

A very useful tool in the study of these models when $G$ is planar is the planar duality of the FK model. For a planar graph $G$, let $G^*$ be its dual graph. To every FK configuration $\omega$, we associate the dual configuration $\omega^*$ given by $\omega^*(e^*)=1$ if and only if $\omega(e)=0$ (where $e^*$ is the dual-edge intersecting $e$). A simple calculation yields that at the self-dual point $p_{\sd}(q) = \frac{\sqrt q} {1+\sqrt q}$, for $G\subset \mathbb Z^2$, we have $\pi_{G,p_{\sd},q}(\omega)=\pi_{G^*,p_{\sd},q} (\omega^*)$ .

In the presence of boundary conditions $\xi$ on $G$ whose augmented graph $G'$ is planar, the same holds for the corresponding dual boundary conditions $\xi^*$ on $G^*$; these are determined on a case by case basis via the planarity of the augmented graph, but importantly, the wired and free boundary conditions are dual to each other.

Throughout the paper, for two vertices $x,y$ we will write $x\stackrel{D}\longleftrightarrow y$ if they are in the same component in $\omega\restriction_{D}$; when we include an asterisk, we mean $x$ and $y$ are in the same dual-component, i.e.\ they are in the same component of the dual configuration $\omega^*$.

\subsubsection*{FKG inequality and monotonicity}When $q\geq 1$, the FK model satisfies positive correlation (FKG) inequalities: if $A$ and $B$ are increasing events in the edge configuration, for any boundary condition $\xi$, we have $\pi^\xi_{G,p,q}(A\cap B)\geq \pi^\xi_{G,p,q}(A)\pi^\xi_{G,p,q}(B)$.

This yields monotonicity in boundary conditions at $q\geq 1$: for FK boundary conditions $\xi'\geq \xi$  ($\xi$ is a finer partition than $\xi'$) and every increasing $A$, $\pi^{\xi'}_{G,p,q}(A) \geq \pi_{G,p,q}^\xi(A)$.

\subsubsection*{The domain Markov property} The FK model also satisfies the domain Markov property: for a graph $G$ with boundary conditions $\xi$, and a subgraph $G'\subset G$,  
\begin{align*}
\pi_G^\xi(\omega \restriction_{G'}\in \cdot \mid \omega\restriction_{G-G'} = \eta) \stackrel{d}= \pi_{G'}^{\eta,\xi}
\end{align*}
where $(\eta,\xi)$ is the boundary conditions induced on $G'$ by connections from $\eta$ and $\xi$. 

\subsubsection*{The Potts and FK phase transition} On $\mathbb Z^2$, the FK and Potts models undergo a phase transition from---in the FK setting--existence a.s.\  of an infinite cluster at $p>p_c$ to a.s.\ no infinite cluster at $p<p_c$. In \cite{BeDu12} it was proved that  for all $q\geq 1$, $p_c(q)=p_{\sd} (q)$. This corresponds to a Potts phase transition from a unique infinite-volume Gibbs measure when $\beta<\beta_c$ to $q$ different extremal Gibbs measures corresponding to weak-limits of boundary conditions of the $q$ colors when $\beta>\beta_c$.

While for $q\leq 4$, the phase transition described above is continuous \cite{DST15}, when $q>4$, the phase transition is discontinuous \cite{DGHMT16}: there are two extremal FK Gibbs measures at $p=p_c$ corresponding to the wired and free boundary conditions at infinity, $\pi^1_{\mathbb Z^2}\neq \pi^0_{\mathbb Z^2}$ (resp.\ at $\beta=\beta_c$, $q+1$ extremal Potts Gibbs measures corresponding to the $q$ different colors, along with a disordered phase with free boundary conditions at infinity). As a result, we have the following~\cite{DST15,DGHMT16}: let $q>4$ and $p=p_c$; there exists $c(q)>0$ such that
\begin{align}\label{eq:exp-decay}
\pi^0_{\mathbb Z^2} \left(0\longleftrightarrow \partial (\llb -n,n\rrb ^2)\right) \lesssim e^{-cn}\,.
\end{align}

\subsubsection*{Surface tension}

The order-disorder surface tension will play a large role in both upper and lower bounds studying the effect of boundary conditions on mixing in the phase coexistence regime. In the low-temperature regime, the metastable phases are the $q$ ordered ones and there is a positive order-order surface tension (see,  e.g., the $q=2$ case); this leads to sensitivity of mixing times to boundary conditions in Ising/Potts Glauber dynamics (cf., e.g.,~\cite{Martinelli94,MaTo10}), but not in FK/Swendsen--Wang dynamics (as these are symmetric w.r.t.\ the $q$ ordered phases). At the critical point, the disordered phase is also metastable and induces similar sensitivity to boundary conditions in FK Glauber and Swendsen--Wang dynamics.

Let $\mathcal S_n = \llb 0,n\rrb \times \llb -\infty,\infty\rrb$ and let $(1,0,\phi)$ FK boundary conditions on $\partial \mathcal S_n$ denote those that are wired on $\partial \mathcal S_n \cap \{(x,y)\in \mathbb Z^2:y\geq x \tan \phi\}$ and free elsewhere on $\partial \mathcal S_n$. We will always be taking $\phi \in (-\pi/2, \pi/2)$. 

\begin{definition}\label{def:surface-tension}
The order-disorder surface tension on $\mathcal S_n$ in direction $\phi$ is given by 
\begin{align*}
\tau_{1,0} ( \phi) =   \lim_{n\to\infty} \frac {\cos \phi}{\beta_c n} \log \bigg[\frac{\mathcal Z^{1,0,\phi}_{\mathcal S_n} }{(\mathcal Z^{1}_{\mathcal S_n}\mathcal Z^{0}_{\mathcal S_n})^{1/2}}\bigg]
\end{align*}
(whenever this limit exists) where $\mathcal Z^\xi_{\mathcal S_n}$ denotes the FK partition function on $\mathcal S_n$ with boundary conditions $\xi$ on $\partial \mathcal S_n$. 
\end{definition}
It was proved first in \cite{LMMRS91} that for $q$ sufficiently large, at $\beta=\beta_c$ and $\phi \neq \pm \frac \pi 2$, the surface tension $\tau_{1,0}(\phi)$ exists and satisfies $\tau_{1,0}(\phi)>0$. In \S\ref{sec:surface-tension-estimates}, we will study consequences of positive surface tension for large deviation estimates of FK interfaces (see Def.~\ref{def:fk-interface}).

\subsection{Markov chain mixing times} We introduce the relevant dynamical quantities and techniques in the study of mixing times; for an extensive treatment, see~\cite{LPW17}.

\subsubsection*{Mixing times} Consider a Markov chain with finite state space $\Omega$, reversible w.r.t.\ an invariant measure $\pi$. Define the total variation distance between measures $\mu,\nu$ on $\Omega$ as
\[\|\mu-\nu\|_\tv=\tfrac 12 \|\mu-\nu\|_{\ell_1}=\sup_{A\subset \Omega} |\mu (A)-\nu(A)|\,,
\]
also defined as a coupling distance $\|\mu-\nu\|_{\tv} =\inf\big\{\mathbb P(X_t\neq Y_t):X\sim \mu,Y\sim \nu\big\}$,
where the infimum is over all couplings $(X,Y)$. Then for a discrete time Markov chain with transition kernel $P(\cdot,\cdot)$, 
we define the total variation mixing time of the chain as
\[\tmix = \inf \Big\{t:\max_{x\in \Omega} \|P^t(x,\cdot)-\pi\|_\tv< 1/(2e)\Big\}\,.
\]

\subsubsection*{Spectral gap and Dirichlet form}
One commonly used technique to bound the mixing time of a Markov chain with transition kernel $P$ is to bound the \emph{spectral gap} of $P$. A transition matrix $P$ reversible w.r.t., $\pi$ has largest eigenvalue $1$ by Perron--Frobenius, and has real spectrum. Thus we can enumerate its spectrum $1>\lambda_1\geq\lambda_2\geq...$ and define its spectral gap as $\gap=1-\lambda_1$; the following  relation with $\tmix$ is well known:
\begin{align}\label{eq:gap-tmix}
\gap^{-1} -1\leq \tmix \leq \log(\tfrac{2e}{\pi_{min}}) \gap^{-1}\,,
\end{align}
where $\pi_{min} = \min_{x\in \Omega} \pi (x)$. The variational form of the spectral gap is
\begin{align}\label{eq:dirichlet-form}
\gap= \inf_{f: f\not \equiv 0, \, \mathbb E f= 0} \frac {\mathcal E(f,f)} {\mathbb E_\pi (f^2)}
\end{align}
where the Dirichlet form $\mathcal E(f,f)$ is given by $\sum_{x,y\in \Omega} \pi(x)P(x,y) (f(y)-f(x))^2$.

Though our main results are stated for the Swendsen--Wang dynamics for the Potts model, as remarked earlier, the proofs all analyze instead the mixing time of the Glauber dynamics for the FK model. We formally define both of these dynamics in the sequel.

\subsubsection*{Heat-bath Glauber dynamics}
Discrete time heat-bath Glauber dynamics for the FK model (FK Glauber dynamics) on a finite graph $G$ with boundary conditions $\xi$ is the Markov chain $(X_t)_{t\geq 0}$ defined as follows: at time $t$ it picks an edge $e$ uniformly at random, and resamples $\omega(e)$ via  $\pi_G^\xi(\omega\restriction_{\{e\}}\mid \omega\restriction_{E-\{e\}}=X_{t-1}\restriction_{E-\{e\}})$ to obtain $X_{t+1}$.  Though this paper does not consider it, the heat-bath Potts Glauber dynamics $(X_t')_{t\geq 0}$ is defined similarly, where at each time step a site $v$ is picked uniformly at random and $\sigma(v)$ is resampled according to $\mu_G^{\eta}(\sigma \restriction_{\{v\}} \mid \sigma \restriction_{V-\{v\}} = X'_{t-1} \restriction_{V-\{v\}})$. 

\subsubsection*{Swendsen--Wang dynamics}
Theorems~\ref{mainthm:1}--\ref{mainthm:3} all treat the Swendsen--Wang dynamics~\cite{SW87}, which, for the $q$-state Potts model on a finite graph $G=(V,E)$ at inverse temperature~$\beta$, is the following discrete-time Markov chain. Given that the Markov chain is at state  $\sigma\in [q]^V$ at time $t$, generate state $\sigma'\in [q]^V$ at time $t+1$ as follows.
\begin{enumerate}
\item To $\sigma$, assign an (random) FK configuration $(\omega(e))_{e\in E}$ defined as follows: for every edge $e=(x,y)$, independently, set $\omega(e)$ to be closed with probability $1$ if $\sigma_x\neq \sigma_y$ and with probability $e^{-\beta}$ if $\sigma_x = \sigma_y$.
\item For every connected component of $\omega$,
reassign all the vertices in the cluster, collectively, an i.i.d.\ color in $[q]$, to obtain the new configuration $\sigma'$.
\end{enumerate}
One can check using the Edwards--Sokal coupling of the FK and Potts models, that the Swendsen--Wang dynamics is reversible with respect to $\mu_{G,\beta}$.

\begin{remark}
The definition of the Swendsen--Wang dynamics indicates why it has fast mixing at low-temperature whereas the Potts Glauber dynamics slows down on $(\mathbb Z/n\mathbb Z)^2$ (see Theorem~4 of~\cite{GL16a}). When $\beta>\beta_c$, the metastable states are the $q$ ordered phases, corresponding to the $q$ colors; but step 2) above recolors all FK clusters, and in particular, reassigns the large macroscopic cluster of the ordered phase an i.i.d.\ color, allowing the dynamics to easily jump between the $q$ metastable states. However, as it relies heavily on the FK representation of the model, the order-disorder energy barrier is still hard to overcome when both ordered and disordered phases are metastable.     
\end{remark}

\subsubsection*{Comparison between cluster and Glauber dynamics}
The following  estimates allow us to reduce the analysis of the Swendsen--Wang dynamics to the analysis of FK Glauber dynamics, and demonstrate why our results do not carry over to  Glauber dynamics for the Potts model.

\begin{theorem}[\cite{Ul13,Ul14}]\label{thm:Ullrich-comparison}
Let $q \geq 2$ be integer.
Let $\gap_{\textsc p}$ and $\gap_{\rc}$ be the spectral gaps of discrete-time Glauber dynamics for Potts and FK model on a graph on $m$ edges and maximum degree $\Delta$, resp., and let $\gap_{\textsc {sw}}$
be the
spectral gap of Swendsen--Wang.
Then
\begin{align}
\gap_{\textsc p} &\leq  2q^2 (qe^{2\beta})^{4\Delta}\gap_{\textsc {sw}}\,,\label{eq-ullrich1}\\
\gap_{\rc} &\leq\gap_{\textsc {sw}}\leq 16 \gap_{\rc}\, m\log m \,.\label{eq-ullrich2}
\end{align}
\end{theorem}

Theorem~\ref{thm:Ullrich-comparison} implies, in particular, that the mixing time of the FK Glauber dynamics and Swendsen--Wang dynamics on a graph $G$ are comparable up to polynomial factors in the number of vertices.  However, the Potts Glauber dynamics can be much slower than both, e.g., on the $n\times n$ torus when $\beta>\beta_c(q)$.  

\subsubsection*{Potts boundary conditions and mixing}
The comparison estimates above are only valid in the absence of boundary conditions, while the focus of this paper is the influence of various fixed boundary conditions.
In particular, the estimates of Theorem~\ref{thm:Ullrich-comparison} hold immediately for $\Lambda_{n,n}$ with free or periodic (the torus) boundary conditions. For other boundary conditions, we can deform $\Lambda_{n,n}$ as in the following remark; this could, however, distort $\Delta$ by an order $n$ factor, leading to an exponential in $n$ cost in~\eqref{eq-ullrich1}.

\begin{remark}\label{rem:bc-deform}
For any FK boundary condition $\xi$ on $G$, we can define a (not necessarily planar) graph $\tilde G$ by identifying all vertices of every boundary component of $\xi$ with a single vertex in $\tilde G$, and keeping the same edge structure. Then the FK Glauber dynamics on $\tilde G$ with free boundary conditions is the same as that on $G$ with boundary conditions $\xi$. In such a case, let $\partial \tilde G$ denote the set of vertices in $\tilde G$ that arise from the boundary components of $G$.
\end{remark}

\begin{remark}\label{rem:max-degree-prime}
The exponential dependence on $\Delta$ in Eq.~\eqref{eq-ullrich1} can be improved to exponential in the maximum degree of all but one vertex (see~\cite[Theorem~1']{Ul13}), whence $\Delta$ in Eq.~\eqref{eq-ullrich1} can be replaced with the second largest degree of $G$. This implies that Theorem~\ref{mainthm:1} in fact also holds for the Potts Glauber dynamics.
\end{remark}

The following is a consequence of spin symmetry of the Swendsen--Wang dynamics. 

\begin{fact}\label{fact:one-color}
Consider  Swendsen--Wang dynamics on $G$ with Potts boundary conditions $\eta$, by considering the graph $\tilde G$, where boundary vertices of each color are identified as single vertices. Let $\gap_{\textsc{sw},\eta,G}$ be the spectral gap of Swendsen--Wang dynamics on $G$ with b.c., $\eta$, and let $\gap_{\textsc{sw},\eta,\tilde G}$  be the spectral gap of Swendsen--Wang dynamics on $\tilde G$ with $\partial \tilde G$ assigned the colors given by $\eta$. If $\partial \tilde G$ consists of at most one vertex, then
\[\gap_{\textsc{sw},\eta,G}=\gap_{\textsc{sw},\eta,\tilde G}=\gap_{\textsc{sw},0,\tilde G}\,.
\]
\end{fact}

Remark~\ref{rem:max-degree-prime} and Fact~\ref{fact:one-color} imply that for $\Lambda_{n,n}$ with FK boundary conditions $\xi$ with at most one nontrivial boundary component, corresponding to Potts b.c., $\eta$ (an assignment of a color, e.g., red, to the nontrivial component of $\xi$ and no other color assignments),
\[\gap_{\rc} \lesssim\gap_{\textsc{sw}}\lesssim n^2 \log n \cdot \gap_{\rc}\,.
\]
Notice that all of the boundary conditions considered in Theorems~\ref{mainthm:1}--\ref{mainthm:3} are of this form. Thus, it suffices to prove Theorems~\ref{mainthm:1}--\ref{mainthm:3} for the FK Glauber dynamics. From now on, when we write $\gap$ with no subscript we mean it to refer to the FK Glauber dynamics. 

\subsubsection*{Monotonicity and the grand coupling}  
A discrete-time Markov chain with state space $\Omega$ and transition kernel $P$ is \emph{monotone} if $\mu P \succeq \nu P$ for every two probability distributions $\mu,\nu$ on $\Omega$ such that $\mu \succeq \nu$; we say that a spin system is monotone whenever single-site Glauber dynamics for it is monotone (cf.~\cite[\S22.3]{LPW17}).

The random mapping representation of the discrete time FK Glauber dynamics on a graph $G=(V,E)$ views the updates as a sequence $(J_i,U_i)_{i\geq 1}$, in which $J_i$'s are i.i.d.\ uniform edges (the updated locations), and the $U_i$'s are i.i.d.\ uniform on $[0,1]$: starting from an initial configuration $\omega_0$, at time $T_i$, writing $J_i=(x,y)$, the dynamics replaces the value of $\omega(J_i)$ by $\one\{U_i\leq p\}$ if $x\longleftrightarrow y$ in $E-\{J_i\}$ and by $\one\{U_i\leq \frac{p}{p+q(1-p)}\}$ otherwise.
The \emph{grand coupling} for FK Glauber dynamics is a coupling of the chains from all initial configurations on $G$ which, via the above random mapping representation, uses the same update sequence $(J_i,U_i)_{i\geq 1}$ for each one of these chains. Using this representation and the FKG inequality, one sees that heat-bath Glauber dynamics for the FK model at $q\geq 1$ is monotone: for every two FK configurations $\omega_1\ge \omega_2$ and every $t\ge 0$, we have $P^t(\omega_1,\cdot)\succeq P^t(\omega_2,\cdot)$.

Thus, for $q\geq 1$, this coupling preserves the partial ordering of the Markov chains started from all possible initial configurations, at all times $t\geq 0$. In particular, under the grand coupling, the value of an edge $e$ in Glauber dynamics at time $t$ from an arbitrary initial state $\omega_0$, is sandwiched between the corresponding values from the free and wired initial states; thus, by a union bound over all edges,
\begin{align}\label{eq:init-config-comparison}
\max_{x\in \Omega} \|P^t (x,\cdot)-\pi\|_\tv \leq \max_{x,y\in\Omega}\|P^t(x,\cdot)-P^t(y,\cdot)\|_{\tv} \leq  |E(\Lambda)|\, \|P^t(1,\cdot)-P^t(0,\cdot)\|_{\tv}\,.
\end{align}

\subsubsection*{Censoring inequalities}

The Peres--Winkler censoring inequalities~\cite{PW13} for monotone spin systems allow one to ``guide'' the dynamics to equilibrium, using that prohibiting updates at various sites for periods of time can only slow down mixing.

\begin{theorem}[{\cite[Theorem~1.1]{PW13}}]\label{thm:censoring}
Let $\mu_T$ be the law of the discrete-time Glauber dynamics at time $T$ for a monotone spin system with state space $\Omega=S^\Lambda$ and stationary distribution $\pi$, whose initial distribution $\mu_0$ is such that $\omega\mapsto\mu_0(\omega)/\pi(\omega)$ is increasing w.r.t.\ the partial order on $\Omega$. Set $0=t_0 < t_1 <\ldots < t_k = T$ for some $k$, let $(\Lambda_i)_{i=1}^k$ be subsets of $\Lambda$, and let $\tilde \mu_T$ be the law at time $T$ of the censored dynamics, started at $\mu_0$, where only updates in $\Lambda_i$ are kept during the times $\{t_{i-1}+1,\ldots,t_i\}$. Then $\|\mu_T-\pi\|_\tv \leq \|\tilde\mu_T-\pi\|_\tv$ and $\mu_T \preceq \tilde\mu_T$; moreover, both $\mu_T/\pi$ and $\tilde\mu_T / \pi$ are increasing.
\end{theorem}
Theorem~\ref{thm:censoring} implies in particular that if we consider two chains---the standard FK Glauber dynamics $X_t\sim P^t(1,\cdot)$, and the modified dynamics $\tilde X_t \sim \tilde P^t(1,\cdot)$ which at certain times sets all bonds in predetermined subsets of $G$ to be wired---then $\tilde P^t(1,\cdot)$ is always farther from $\pi$ and from $P^t(0,\cdot)$ in total-variation than $P^t(1,\cdot)$ is.

\section{Consequences of positive surface tension}\label{sec:surface-tension-estimates}

\subsection{Cluster expansion}\label{subsec:cluster-expansion}
We first introduce the cluster expansion framework used to prove that the FK order-disorder surface tension $\tau_{1,0}(\phi)$ as defined in Definition~\ref{def:surface-tension} is positive for all large $q$ at $p=p_c(q)$. We skip many of the details here as understanding them is not necessary to the equilibrium estimates we require (\S\ref{subsec:surface-tension-estimates}). This approach was extensively developed in \cite{DKS} for the low-temperature Ising model when $\beta \gg\beta_c$, and then extended to the critical Potts/FK model for large $q$ (at the discontinuous phase transition point) in~\cite{MMRS91} where one can find more details on the below. Though these cluster expansion techniques only go through at $p=p_c$ when $q$ is sufficiently large, all of the interface estimates are expected to hold at $p_c$ whenever $\tau_{1,0}(\phi)>0$, so in particular, for \emph{every} $q>4$.

In what follows, two (primal) edges $e,f$ are adjacent if they share a vertex. Two primal edges $e,f$ are co-adjacent if $e^\star$ is adjacent to $f^\star$. Connectedness and co-connectedness are defined naturally with respect to adjacency and co-adjacency. 
For an edge subset $A\subset E(\mathbb Z^2)$, the (edge)-\emph{boundary of $A$} is the set of all edges in $A$ that are co-adjacent to $E(\mathbb Z^2) \setminus A$. The \emph{co-boundary of $A$} is the set of edges in $E(\mathbb Z^2)\setminus A$ that are adjacent to $A$.    

\begin{definition}\label{def:fk-interface}
Let $D$ be a connected subgraph of $\mathbb Z^2$ and let $a,b$ be two marked boundary points on $\partial D$. Consider Dobrushin boundary conditions that are wired on the clockwise segment $(a,b)$ and free on $(b,a)$ (where if $D$ is infinite, simply connected, we define these boundary arcs in the natural way). Then for an FK realization $\omega$ on $D$, the \emph{primal-FK interface} (or simply the FK interface), $\mathcal I= \mathcal I(\omega)$, is defined as follows: 
\begin{enumerate}
\item Consider the dual-component of the boundary arc $(b,a)$ in $\omega$, and consider its co-boundary (a set of closed dual-edges in $\omega$). 
\item This co-boundary has a unique co-connected component, call it $\Gamma^\star(\omega)$, that is incident to the boundary arc $(a,b)$. The primal-FK interface $\mathcal I$ is the set of all primal edges that are dual to edges in $\Gamma^\star (\omega)$.  
\end{enumerate}
(Notice that this interface is not necessarily a simple path, but it is connected, and will have no cycles.)
 Define the \emph{dual-FK interface} analogously as follows: consider the connected component of open edges touching the boundary arc $(a,b)$ and consider its co-boundary (a set of closed primal edges). This co-boundary has a unique co-connected component that is incident the boundary arc $(b,a)$; the set of dual-edges that are dual to edges in this co-connected component will be the dual-FK interface.  For more general boundary conditions, there will be a compatible collection of interfaces between all boundary segments that are wired.
\end{definition}

Recall that we defined the infinite strip $\mathcal S_n = \llb 0,n\rrb \times \llb -\infty,\infty\rrb$. 

\begin{definition}For any angle $\phi\in (-\frac \pi 2, \frac \pi 2 )$, an \emph{edge-cluster weight function} $\Phi(\mathcal C,\mathcal I)$ is any real-valued function with first argument that is a connected set of edges in $\mathcal S_n$ and second argument that is a possible realization of an FK interface with Dobrushin boundary conditions that are wired on the clockwise arc between  $(0,0)$ to $(n,\lfloor n\tan \phi \rfloor)$, such that for some $\lambda>0$ and every $\mathcal C, \mathcal I$, 
\begin{enumerate}
\item $\Phi(\mathcal C,\mathcal I)=0$ when $\mathcal C\cap \mathcal I=\emptyset$\,,
\item $\Phi(\mathcal C,\mathcal I_1)=\Phi(\mathcal C,\mathcal I_2)$ when $\Pi_{\mathcal C} \cap \mathcal I_1=\Pi_{\mathcal C}\cap \mathcal I_2$\,,
\item $\Phi(\mathcal C,\mathcal I)=\Phi(\cC+(0,s),\mathcal I_1)$ when $\mathcal I_1\cap \Pi_{\mathcal C}=(0,s)+\mathcal I\cap \Pi_{\mathcal C}$\,,
\item $|\Phi(\mathcal C,\mathcal I)|\leq \exp(-\lambda  \ell(\mathcal C))$\,,
\end{enumerate}
 where $$\Pi_{\mathcal C}=\{(x,y)\in \mathbb R^2:\exists y' \mbox{ s.t.\ }(x,y')\in \mathcal C\}\,,$$ and $\ell(\mathcal C)$ is the minimum number of edges in a connected subset of $E(\Z^2)$ that contains all the boundary edges of $\mathcal C$. 
More generally, for an edge-cluster weight function $\Phi$ and a domain $V_n$ with Dobrushin boundary conditions between $(0,0)$ and $(n,n\lfloor \tan \phi \rfloor)$, its partition function is given by
\begin{align}\label{eq:cluster-weight-function}
\cZ_\Phi = \mathcal Z_\Phi(V_n, \phi) = \sum_{\mathcal I} \lambda^{|\mathcal I|+\sum_{\mathcal C: \mathcal C\cap \mathcal I = \emptyset} \Phi(\mathcal C, \mathcal I)}\,,
\end{align}
where the sum runs over all possible FK order-disorder interfaces in $V_n$. 

There exists a unique edge-cluster weight function $\Phi_{\textsc{o/d}}$---called the \emph{FK order-disorder weight function}---such that the probability that the FK interface in $\cS_n$ under the $(1,0,\phi)$ boundary condition is $\cI$ is given by
\[ \pi_{\cS_n}^{(1,0,\phi)}(\cI) = \cZ_{\Phi_{\textsc{o/d}}}^{-1}\, \lambda^{|\mathcal I| + \sum_{\mathcal C:\mathcal C\cap \mathcal I \neq \emptyset} \Phi_{\textsc{o/d}}(\mathcal C,\mathcal I)}\,.\]
The FK order-disorder weight function $\Phi_{\textsc{o/d}}$ is made explicit in~\cite[Proposition~5]{MMRS91}. 
\end{definition}

By adapting the methods of~\cite{DKS} to the FK cluster expansion, the following large deviation estimate on order-disorder interface fluctuations was obtained in~\cite{MMRS91}.

\begin{proposition}[{\cite[Proposition 5]{MMRS91}}] \label{prop:fk-415} Consider the critical FK model on $\mathcal S_n$ and fix a $\delta>0$. There exist some $q_0$ and $c(\delta)>0$ such that for all $\phi\in [-\frac \pi 2+\delta ,\frac \pi 2-\delta ]$ and every $q\geq q_0$,  every $h\geq 1$ and $n\geq 1$,
\[\pi_{\mathcal S_n}^{1,0,\phi}\big(\mathcal I\not\subset \{(x,y):y\in [  x \tan \phi -h, x\tan \phi + h] \}\big)\lesssim n^2 \exp(-ch^2/n)\,.
\]
\end{proposition}

The following is a finer result, that reformulates results of \cite{DKS} in the FK setting.
Define the \emph{cigar-shaped region} $U_{\kappa,d,\phi}$ for every $d>0$ and $\kappa>0$ by
\begin{equation}\label{eq:cigar-shaped-region}
U_{\kappa,d,\phi}=\mathcal S_n\cap \left\{ (x,y)\in \mathbb Z^2:|y-x\tan \phi| \leq d\left|\tfrac {x(n-x)}n\right|^{\frac 12+\kappa}\right\}\,,
\end{equation}

\begin{proposition} \label{prop:fk-416} Consider the critical FK model on a domain $V_n\supset U_{\kappa,d,\phi}$, let $\Phi_{\textsc{o/d}}$ be the FK order-disorder weight function,  and let $\tilde \Phi$ be any function satisfying
\[\tilde \Phi(\mathcal C,\mathcal I)=\Phi_{\textsc{o/d}}(\mathcal C,\mathcal I) \mbox{ when } \mathcal C\subset U_{\kappa,d,\phi}\,, \qquad | \tilde \Phi(\mathcal C,\mathcal I)|\leq \exp(-\lambda \ell(\mathcal C)) \qquad \mbox{ for every $\mathcal C, \mathcal I$}
\] (e.g., $\tilde \Phi=\Phi_{\textsc{o/d}} \one\{\cC\subset U_{\kappa,d,\phi}\}$). There exist  $q_0>0$ and $f(\kappa)=O(\kappa^{-1})$ such that for all $q\geq q_0$, all $\phi \in [-\frac{\pi}2 +\delta, \frac{\pi}{2}-\delta]$, there exists $C(q,\delta)>0$, such that
\begin{align}\label{eq:domain-comparison}
|\log {\mathcal Z}_{\tilde \Phi}(V_n,\phi)-\log \mathcal Z_{\Phi_{\textsc{o/d}}} (\mathcal S_n,\phi)|\leq C(\log n)^{f(\kappa)}\,.
\end{align}
Moreover, the order-disorder surface tension $\tau_{1,0}(\phi)$ (given in Def.~\ref{def:surface-tension}) satisfies
\begin{align}\label{eq:surface-tension}
|\log { {\mathcal Z}}_{\tilde \Phi} (V_n,\phi)-n\log (1+\sqrt q) (\cos\phi)^{-1}\tau_{1,0}(\phi)|\leq  C(\log n)^{f(\kappa)}\,,
\end{align}
and the large deviation estimate of Proposition~\ref{prop:fk-415} holds for $V_n$ and $d/2$.
\end{proposition}

\subsection{Surface tension estimates on subsets of $\Lambda_{n,n}$}\label{subsec:surface-tension-estimates}
In this section, we extend the surface tension estimate of Proposition~\ref{prop:fk-415} to tilted half-infinite and finite strips. We first define subsets of $\mathcal S_n = \llb 0,n\rrb \times \llb -\infty, \infty\rrb$ we consider in obtaining the desired sub-exponential upper bounds of Theorems~\ref{mainthm:1}--\ref{mainthm:2} (see Figure~\ref{fig:subsets}). 

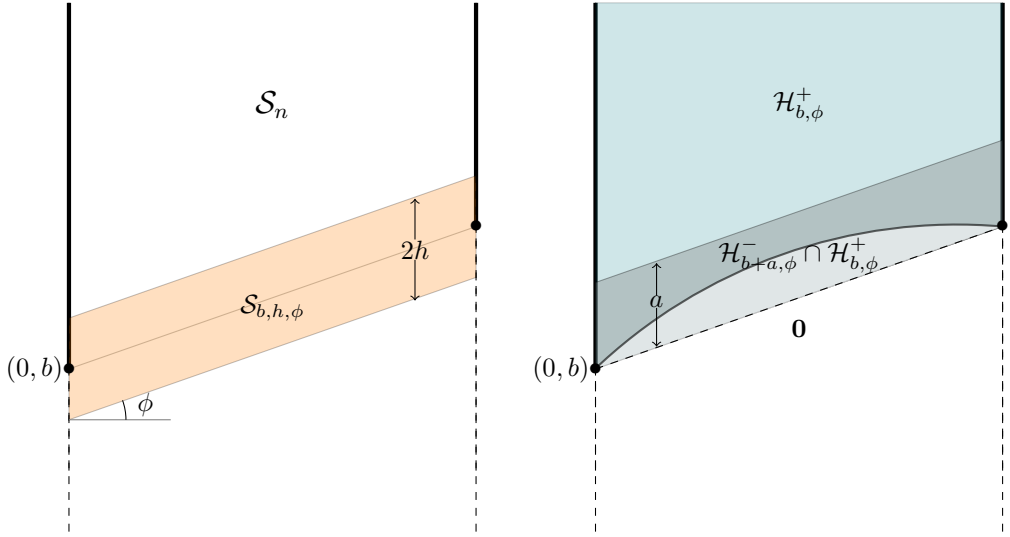
\begin{figure}
  \begin{tikzpicture}
    \node (plot1) at (0,0) {};

    \node (plot2) at (7,0) {};

    \begin{scope}[shift={(plot1.south west)},x={(plot1.south east)},y={(plot1.north west)}, font=\small]

\draw[color=black, dashed] (0,0) -- (0,23);
\draw[color=black, dashed] (20,0)--(20,23);
\draw[draw=black, ultra thick] (0,5) -- (0,23);
\draw[draw=black, ultra thick] (20,23)--(20,12);

\draw[draw=lightgray] (0,5)--(20,12);
\draw[fill=orange, opacity=.25] (0,5)--(0,7.5)--(20,14.5)--(20,9.5)--(0,2.5)--(0,5);

\draw[color=black, dashed] (20,-3)--(20,23);

\draw[color=black, dashed] (0,-3) --(0, 23);

\draw[color=gray] (0,2.5)--(5,2.5);
\node[font=\small] at (3.75,3.25) {$\phi$};

\draw[<-] (17,8.365)--(17,10.365); \node[font=\small] at (17,10.865) {$2h$}; \draw[->] (17, 11.365)--(17,13.365);

    \node[font=\small] at (-1.7,5) {$(0,b)$};  
    \node[font=\small] at (10,8) {$\mathcal S_{b,h,\phi}$};
        \node[font=\large] at (10,18) {$\mathcal S_n$};
        
        \draw (2.8, 2.5) arc (0:20:2.8);

     \node[font=\small] at (0,5) {$\bullet$};
     \node[font= \small] at (20,12) {$\bullet$};

    \end{scope}

    \begin{scope}[shift={(plot2.south west)},x={(plot2.south east)},y={(plot2.north west)}, font=\small]

\draw[color=black, dashed] (0,0) -- (0,23);
\draw[color=black, dashed] (20,0)--(20,23);
\draw[draw=black, ultra thick] (0,5) -- (0,23);
\draw[draw=black, ultra thick] (20,23)--(20,12);

\draw[draw=lightgray] (0,5)--(20,12);

\draw[color=black, dashed] (20,-3)--(20,23);

\draw[color=black, dashed] (0,-3) --(0, 23);

    \node[font=\small] at (-1.7,5) {$(0,b)$};

\draw[fill=green, opacity=.1] (0,5)--(0,23)--(20,23)--(20,12)--(0,5);
\draw[fill=blue, opacity=.1] (0,5)--(0,23)--(20,23)--(20,12)--(0,5);

\draw[fill=gray, opacity=.35] (0,5)--(0,9.25)--(20,16.25)--(20,12)--(0,5);

\node[font=\small] at (10,7) {$\mathbf{0}$};

\draw[ color=black, thick, fill=white, opacity=.6] (20,12) arc (85:134:25.8) ;

\node[font=\small] at (10,18) {$\mathcal H_{b,\phi}^+ $};

\node[font=\small] at (10,10.6) {$\mathcal H_{b+a,\phi}^- \cap \mathcal H_{b,\phi}^+ $};
\draw[<-] (3,6.1)--(3,8); \node[font=\small] at (3,8.3) {$a$}; \draw[->] (3, 8.6)--(3,10.2);

     \node[font=\small] at (0,5) {$\bullet$};
     \node[font= \small] at (20,12) {$\bullet$};
     
     \draw[color=black, dashed] (0,5)--(20,12);

    \end{scope}
  \end{tikzpicture}
  \caption{An illustration of some of the sets and interfaces considered in Definition~\ref{def:strips} and Proposition~\ref{prop:surface-tension}. The $(1,0,\phi)$ boundary conditions are depicted with bold lines being wired and dashed lines being free.}\label{fig:subsets}
  \end{figure}

\begin{definition}\label{def:strips}
For every $n$, which will be understood by context and omitted from the following notation, define a tilted strip  $\mathcal S_{b,h,\phi}$ as (for $b\in \mathbb R,h\in \mathbb R_+$, $\phi\in (-\frac \pi 2,\frac \pi 2)$),
\[\mathcal S_{b,h,\phi}=\{(x,y)\in \mathcal S_n: y\in \llb b-h+x \tan \phi, b+h+x \tan \phi \rrb\}\,.
\]
Also, define the half-infinite strips,
\begin{align*}
\mathcal H^+_{b,\phi} = \{(x,y)\in \mathcal S_n: y\geq b+x \tan \phi\}\,, \qquad \mathcal H^-_{b,\phi} = \{(x,y)\in \mathcal S_n:y\leq b+x \tan \phi\}\,.
\end{align*}
\end{definition}

The following estimate---extending an estimate of \cite{MMRS91}, which in turn adapts~\cite{DKS} to the FK model, to half-infinite strips---is a consequence of monotonicity of the FK model and Proposition~\ref{prop:fk-415}. 

\begin{proposition}\label{prop:surface-tension}
Fix  $\delta>0$ and let $q$ be sufficiently large. For any $b\in \mathbb R$ and $\phi\in [-\frac \pi 2 +\delta,\frac \pi 2 -\delta]$,  consider the critical FK model on $\mathcal H^+_{b,\phi}\subset \mathcal S_n$ with boundary conditions $(1,0,b,\phi)$ denoting wired boundary conditions on $\partial \mathcal H^+_{b,\phi} \cap \mathcal H^+_{b+1,\phi}$  and free boundary elsewhere on $\partial \mathcal H^+_{b,\phi}$. There exist constants $A>0$ and $c(\delta,q)>0$ such that the order-disorder interface $\mathcal I$, satisfies
\[\pi^{1,0,b,\phi}_{\mathcal H^+_{b,\phi}} (\mathcal I\not\subset \mathcal H_{b+a,\phi}^- \cap \mathcal H^+_{b,\phi})\leq A n^2\exp\left(-c a^2 / n\right)\,.
\]
\end{proposition}

\begin{proof}The proposition was proven in the case $\phi=0$ in Proposition 4.2 of \cite{GL16a} combining monotonicity of the FK model in boundary conditions with Proposition~\ref{prop:fk-415}.  The same proof carries over to the case $\phi\neq 0$ as long as $\phi$ is uniformly bounded away from $\pm \frac \pi 2$ as the surface tension estimate of Proposition~\ref{prop:fk-415} on $\mathcal S_n$ is expressed in that setup.
\end{proof}

The following is the main equilibrium estimate we will use in the proofs of sub-exponential mixing for general Dobrushin boundary conditions. 

\begin{proposition}\label{prop:strip-surface-tension}
Fix  $\delta>0$ and let $q$ be sufficiently large. For any $b\in \mathbb R$, $h\leq m\leq n$, and $\phi\in [-\frac \pi 2+\delta,\frac \pi 2 -\delta]$, consider the critical FK model on $\mathcal S_{b,m,\phi}\cap \Lambda_{n,n}$ with $(1,0,b,\phi)$ boundary conditions denoting wired on $\partial (\mathcal S_{b,m,\phi} \cap \Lambda_{n,n})\cap \mathcal H^+_{b+1,\phi}$ and free elsewhere on $\partial (\cS_{b,m,\phi}\cap \Lambda_{n,n})$. Then there exists $c(\delta,q)>0$ such that
\[\pi_{\cS_{b,m,\phi}\cap \Lambda_{n,n}}^{1,0,b,\phi}(\mathcal I \not \subset  \mathcal S_{b,h,\phi}\cap \Lambda_{n,n} )\lesssim n^2 \exp({-ch^2/n})\,.
\]
\end{proposition}

We need the following preliminary estimate (see Figure~\ref{fig:interface-bounds} as a guide).

\begin{figure}
  \begin{tikzpicture}
    \node (plot1) at (0,0) {};

    \node (plot2) at (7,0) {};

    \begin{scope}[shift={(plot1.south west)},x={(plot1.south east)},y={(plot1.north west)}, font=\small]

\draw[color=black, dashed] (0,0) rectangle (20,20);
\draw[draw=black, ultra thick] (0,10) -- (0,20)--(20,20)--(20,0)--(17,0);
\draw[draw=black, ultra thick] (0,20)--(0,23);
\draw[draw=black, ultra thick] (20,-1.765)--(20,23);

\draw[draw=gray] (0,15)--(20,3.235)--(20,0)--(8.5,0)--(0,5)--(0,15);
\draw[fill=orange, opacity=.15] (0,15)--(20,3.235)--(20,0)--(8.5,0)--(0,5)--(0,15);

\draw[draw=gray, opacity=.5] (0,13)--(20,1.235)--(20,0)--(11.9,0)--(0,7)--(0,13);
\draw[fill=blue, opacity=.15] (0,13)--(20,1.235)--(20,0)--(11.9,0)--(0,7)--(0,13);
\draw[fill=white, opacity=.5] (0,0)--(17,0)--(0,10);

\fill[ color=blue, opacity=.2] (8.5,5) to[bend right=10] (17,0) -- (20,0) -- (20,1.235) -- (0,13) -- (0,10) to[bend left=10] (8.5,5);
\fill[ color=white, opacity=.5] (8.5,5) to[bend right=10] (0,10) -- (8.5,5) ;
\draw[ color=black, thick] (8.5,5) to[bend right=10] (0,10) ;
\draw[ color=black, thick] (8.5,5) to[bend right=10] (17,0) ;

\draw[fill=gray, opacity=.22] (0,15)--(20,3.235)--(20,20)--(0,20);
\draw[color=black] (20,-3)--(20,23);

\draw[color=black, dashed] (0,-3) --(0, 23);

\draw[<-] (3,3.235)--(3,5.235); \node[font=\small] at (3,5.735) {$m$}; \draw[->] (3, 6.235)--(3,8.235);
\draw[<-] (5,4.05)--(5,5.05); \node[font=\small] at (5,5.55) {$h$}; \draw[->] (5, 6.05)--(5,7.05);

    \node[font=\small] at (-1,10) {$b$};  
    \node[font=\small] at (10,6) {$\mathcal S_{b,h,\phi}$};
    \node[font=\small] at (17,3) {$\mathcal S_{b,m,\phi}$};

     \node[font=\small] at (0,10) {$\bullet$};
     \node[font= \small] at (17,0) {$\bullet$};

    \end{scope}

    \begin{scope}[shift={(plot2.south west)},x={(plot2.south east)},y={(plot2.north west)}, font=\small]

\draw[color=black, dashed] (0,-3) --(0, 23);
\draw[color=black, dashed] (20,-3)--(20,23);

\draw[color=black, dashed] (0,10)--(17,0)--(20,-1.765);

\draw[color=black, dashed] (0,0) rectangle (20,20);
\draw[draw=black, ultra thick] (0,10) -- (0,20)--(20,20)--(20,0)--(17,0);
\draw[draw=black, ultra thick] (0,20)--(0,23);
\draw[draw=black, ultra thick] (20,-1.765)--(20,23);

    \node[font=\small] at (-1,10) {$b$};  

\draw[fill=orange, opacity=.15] (0,10)--(0,23)--(20,23)--(20,-1.765)--(0,10);
\draw[draw=gray] (0,13)--(20,1.235);

     \node[font=\small] at (0,10) {$\bullet$};
     \node[font= \small] at (17,0) {$\bullet$};
     
\draw[fill=green, opacity=.1] (0,10)--(17,0)--(20,0)--(20,20)--(0,20)--(0,10);
\draw[fill=blue, opacity=.15] (20,-1.765)--(0,10)--(0,20)--(20,20);
\draw[draw=gray, opacity=.3, ultra thick] (0,0)--(17,0)--(0,10)--(0,0);

\node[font=\small] at (14,.7) {$\phi$};
\node[font=\small] at (10,3) {$\mathbf{0}$};

\node[font=\small] at (10,12) {$\mathcal H_{b,\phi}^+\cap \Lambda_{n,n}$};
\draw[<-] (3,8.2)--(3,9.2); \node[font=\small] at (3,9.6) {$a$}; \draw[->] (3, 10.1)--(3,11.2);

\draw[ color=black, thick, fill=white, opacity=.5] (20,-1.765) arc (53.5:65:115) ;

    \end{scope}
  \end{tikzpicture}
  \caption{Left: the sets considered in Proposition~\ref{prop:strip-surface-tension}, where we bound the probability of the interface exceeding height $\pm h$. Right: the proof of Lemma~\ref{lem:subset-surface-tension} uses monotonicity in b.c., to go from bounding the probability of the order-disorder interface exceeding height $a$ in the gray set to the union of the gray and purple sets, to all shaded regions. }\label{fig:interface-bounds}
  \end{figure}
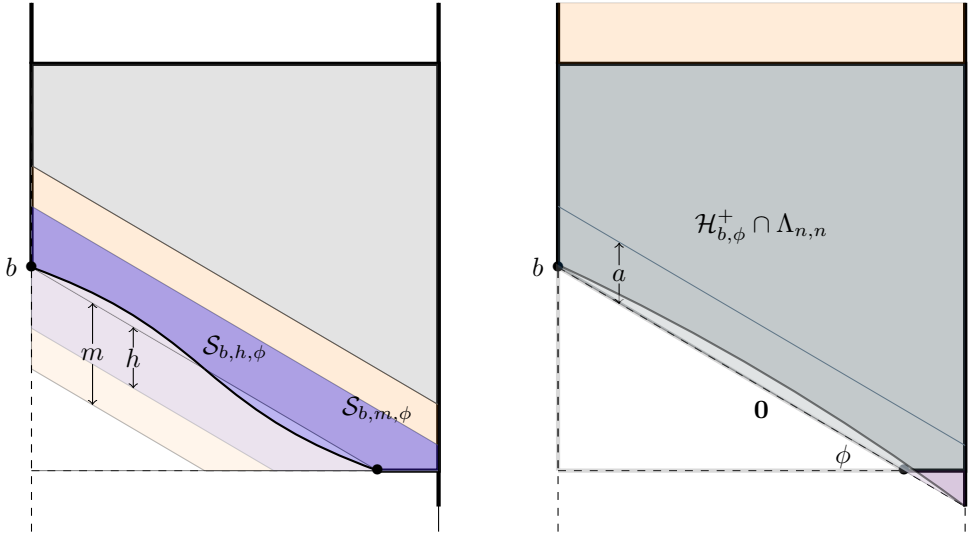

\begin{lemma}\label{lem:subset-surface-tension}
Fix $\delta>0$ and let $q$ be sufficiently large. For any $b\in \mathbb R$ and $\phi\in [-\frac \pi 2 +\delta, \frac \pi 2 - \delta]$,  consider the critical FK model on $\mathcal H^+_{b,\phi}\cap \Lambda_{n,n}$ with $(1,0,b,\phi)$ boundary conditions denoting wired on $\partial (\mathcal H_{b,\phi}^+ \cap \Lambda_{n,n})\cap \mathcal H^+_{b+1,\phi}$ and free elsewhere. There exists $c(\delta,q)>0$ such that
\[\pi_{\mathcal H_{b,\phi}^+ \cap \Lambda_{n,n}} ^{1,0,b,\phi} (\mathcal I \not \subset \mathcal H^-_{b+a,\phi} \cap \mathcal H^+_{b,\phi}\cap \Lambda_{n,n}) \lesssim n^2 \exp({-ca^2/n})\,.
\]
\end{lemma}

\begin{proof}Let $B=\mathcal H^-_{b+a,\phi} \cap \mathcal H^+_{b,\phi}\cap \Lambda_{n,n}$ and consider the region $S'= \mathcal H^+_{b,\phi} \cap \mathcal H^-_{n,0}$.

For a general domain $D\subset \mathcal S_n$ let the boundary conditions $(1,0,b,\phi)$ on it denote wired on $\partial D\cap \mathcal H^+_{b+1,\phi}$ and free elsewhere on $\partial D$. By monotonicity in boundary conditions, and then inclusion of events,

\begin{align*}
\pi _{\mathcal H^+_{b,\phi} \cap \Lambda_{n,n}}^{1,0,b,\phi} (\mathcal I\not \subset B)
& \leq \pi^{1,0,b,\phi} _{S'}(\mathcal I\restriction_{\Lambda_{n,n}} \not \subset B) \\
& \leq \pi^{1,0,b,\phi} _{S'}(\mathcal I \not \subset \mathcal H^-_{b+a,\phi} \cap \mathcal H^+_{b,\phi})\,.
\end{align*}
The application of the FKG inequality here is valid as under $(1,0,b,\phi)$ boundary conditions, events of the sort ``$\mathcal I$ exceeds some height $a$", are decreasing events, as adding edges only pushes the order-disorder interface down. This property will be used throughout the paper. By monotonicity in boundary conditions again,
\[\pi_{S'}^{1,0,b,\phi}(\mathcal I \not \subset \mathcal H^-_{b+a,\phi}\cap \mathcal H^+_{b,\phi}) \leq
\pi_{\mathcal H^+_{b,\phi}}^{1,0,b,\phi} (\mathcal I \not \subset \mathcal H^-_{b+a,\phi}\cap \mathcal H^+_{b,\phi})\,.
\]
By Proposition~\ref{prop:surface-tension}, there exists a $c(\delta,q)>0$ such that
\[\pi_{\mathcal H^+_{b,\phi}}^{1,0,b,\phi} (\mathcal I \not \subset \mathcal H^-_{b+a,\phi}\cap \mathcal H^+_{b,\phi}) \lesssim n^2 e^{-c a^2/n}\,. \qedhere
\]
\end{proof}

\begin{proof}[\emph{\textbf{Proof of Proposition~\ref{prop:strip-surface-tension}}}]
As before, let $(1,0,b,\phi)$ boundary conditions on $D\subset \mathbb Z^2$ denote those that are wired on $\partial D\cap \mathcal H_{b+1,\phi}^+$ and free on $\partial D\cap \mathcal H_{b+1,\phi}^-$.
By a union bound, write
\begin{align}\label{eq:interface-fluctuation}
\pi_{\cS_{b,m,\phi}\cap \Lambda_{n,n}}^{1,0,b,\phi}(\mathcal I \not \subset \mathcal S_{b,h,\phi} )\leq \pi_{\cS_{b,m,\phi}\cap \Lambda_{n,n}}^{1,0,b,\phi}(\mathcal I \not \subset \mathcal H^+_{b-h,\phi} )+ \pi_{\cS_{b,m,\phi}\cap \Lambda_{n,n}}^{1,0,b,\phi}(\mathcal I \not \subset \mathcal H^-_{b+h,\phi} )
\end{align}
and consider the two quantities independently. Observe that the first event on the right-hand side is an increasing event, while the second event is a decreasing event. By reflection symmetry and self-duality, if we prove the desired bound on the latter, for general $b,h,\phi$, it implies the former also.
By monotonicity in boundary conditions,
\begin{align*}
\pi_{\cS_{b,m,\phi}\cap \Lambda_{n,n}}^{1,0,b,\phi} (\mathcal I \not \subset \mathcal H^-_{b+h,\phi}) & \leq \pi _{\cS_{b,m,\phi}\cap \mathcal H^+_{b-1,\phi}\cap \Lambda_{n,n}}^{1,0,b,\phi} (\mathcal I\not \subset \mathcal H^-_{b+h,\phi}) \\
& \leq \pi_{\mathcal H^+_{b-1,\phi}\cap \Lambda_{n,n}} ^{1,0,b,\phi} (\mathcal I\not \subset \mathcal H^-_{b+h,\phi})\,.
\end{align*}
The right-hand side above is exactly the probability bounded in Lemma~\ref{lem:subset-surface-tension}, from which, along with the symmetry noted above and ~\eqref{eq:interface-fluctuation}, the desired upper bound follows.
\end{proof}  

We will also need the following bound in order to prove the mixing time upper bounds on cylinders in \S\ref{sub:subexp-cylinders}. It is an adaptation of the proof of \cite[Lemma A.6]{MaTo10} from the Ising model to the FK model via Propositions~\ref{prop:fk-415}--\ref{prop:fk-416}, and we omit some details.

\begin{proposition}\label{prop:cylinder-midpoint-estimate}
Fix $q$ to be large enough and $\epsilon>0$, and  consider the critical FK model on $\Lambda_{n,h}$ for $n^{\frac 12+\epsilon} \leq h\leq n$ with $1,0$ boundary conditions denoting wired on $\partial_\south \Lambda_{n,h}$ and free elsewhere. For $\rho\in(0,1)$ and $\delta$ small enough, there exists $c(\epsilon)>0$ such that
\[\pi_{\Lambda_{n,h}}^{1,0}(\mathcal I \cap  \llb\lfloor \tfrac n2\rfloor- {\delta n}, \lfloor \tfrac n2\rfloor +{\delta n} \rrb \times \llb 0,\rho h\rrb=\emptyset) \gtrsim e^{-c(\rho h)^2/n}\,.
\]
\end{proposition}

\begin{proof}
Denote by $B= \llb \lfloor \frac n2 \rfloor -\delta n, \lfloor \frac n2 \rfloor +\delta n\rrb \times \llb 0,\rho h\rrb$.
Recall the definition of the cigar-shaped region $U_{\kappa, d,\phi}$ in~\eqref{eq:cigar-shaped-region}.
Following ~\cite{MaTo10}, for every $-\log_2 n +2\leq i \leq \log_2 n-2$, let $z_i$ be the nearest vertex to
\[\left(\tilde x_i, d\left|\frac {\tilde x_i(n-\tilde x_i)}{n} \right|^{\frac 12+\kappa}\right)\qquad \mbox{where}\qquad \tilde x_i= \frac n2 + \frac {\mbox{sgn}\,i}{4} \sum_{j=1}^{|i|-1} 2^{-j}\,.
\]
Let $U_{z_i,z_{i+1}}$ be the cigar shaped region $z_i+U_{\epsilon/2,(1-\rho)\wedge \rho,\phi_{z_i,z_{i+1}}}$ where $d=(1+\rho) h/n^{\frac 12 + \epsilon}$ and $\phi_{z_i,z_{i+1}}$ is the angle of the vector from $z_i$ to $z_{i+1}$ (see, e.g.,~\cite[Fig.~8]{MaTo10}).

By monotonicity in boundary conditions and $\{\mathcal I \cap B =\emptyset\}$ being an increasing event, 
\[\pi_{\Lambda_{n,h}}^{1,0} (\mathcal I \cap B = \emptyset ) \geq \pi_{\mathcal H^-_{h,0}}^{1,0}(\mathcal I \cap B = \emptyset)\,,
\]
where $(1,0)$ boundary conditions on $\partial \mathcal H^-_{h,0}$ are wired on $\partial \mathcal H^-_{h,0} \cap \mathcal H^-_{0,0}$ and free elsewhere. Since $U_{z_i,z_{i+1}}\cap B=\emptyset$ for all $i$, if $\mathcal I\subset \bigcup _{|i|\leq \log_2 n}U_i$, then the desired property holds.

Lemma A.6 of \cite{MaTo10} gives  a lower bound on the partition function restricted to interfaces contained in $\bigcup U_i$ as defined above in the setting of the Ising model; the same proof extends that lower bound to the partition function of such FK interfaces, noting that Proposition~\ref{prop:fk-416} is an analogue of \cite[Theorem 4.16]{DKS}: there exists $c>0$ such that
\[\sum _{\mathcal I \subset \mathcal H^-_{h,0}} \lambda^{|\mathcal I| +\sum_{\mathcal C\cap \mathcal I \neq \emptyset} \Phi(\mathcal C,\mathcal I)}\boldsymbol 1\{\mathcal I \subset \mbox{$\bigcup U_i$}\}\geq \exp(-\beta_c \tau_{1,0} (0)  n - c (\rho h)^2/n)\,,
\]
where an error of $(\log n)^{O(1)}$ was absorbed into the term $c (\rho h)^2/n$ via the assumption that $h\geq n^{1/2+\epsilon}$ and a choice of a suitable constant $c>0$.

Furthermore, Proposition~\ref{prop:fk-416}, in particular ~\eqref{eq:surface-tension}, implies there exists $c>0$ so that
\begin{align*}
\sum _{\mathcal I \subset \mathcal H_{h,0}^-} \lambda^{|\mathcal I|+\sum_{\mathcal C\cap \mathcal I \neq \emptyset} \Phi(\mathcal C,\mathcal I)} & 
\leq \exp( -\beta_c \tau_{1,0}(0) n+ c(\log n)^c)\,.
\end{align*}
Then writing the desired probability as in \cite[(A.36)]{MaTo10} as
\begin{align*}
\pi_{\mathcal H^-_{h,0}}^{1,0}\left(\mathcal I \cap B \neq\emptyset\right) & \geq \pi_{\mathcal H^-_{h,0}}^{1,0} \Big(\mbox{$\mathcal I \subset \bigcup_{|i|\leq \log_2 n-2} U_i $}\Big) \\
& = \frac {\sum_{\mathcal I\subset \mathcal H^-_{h,0}} \lambda^{|\mathcal I|+\sum_{\mathcal C\cap \mathcal I \neq \emptyset} \Phi(\mathcal C,\mathcal I)}\boldsymbol 1\{\mathcal I \subset \bigcup U_i\}}{\sum_{\mathcal I\subset \mathcal H^-_{h,0}} \lambda^{|\mathcal I|+\sum_{\mathcal C\cap \mathcal I \neq \emptyset} \Phi(\mathcal C,\mathcal I)}}\,,
\end{align*}
and plugging in the above lower bound on the numerator and upper bound on the denominator, concludes the proof.
\end{proof}

\section{Sub-exponential mixing with symmetry-breaking boundary}

In this section we prove sub-exponential upper bounds on the inverse spectral gap of Swendsen--Wang dynamics for general Dobrushin boundary conditions. Here, there is a single high probability minimizer of the surface tension for the model, breaking the order-disorder phase symmetry that induces slow mixing in Section~\ref{sec:slow-mixing}. In the proofs in this section, we will need to consider random boundary conditions on $\Lambda_{n,m}$. 

\begin{definition}[wired/free-at-infinity b.c.] \label{def:boundary-condition} Define the \emph{wired-at-infinity} boundary conditions as the distribution over boundary conditions on $\Delta\subset \partial \Lambda_{n,m}$ given by the distribution on partitions of $\Delta$ induced by $\pi^{1}_{\mathbb Z^2}(\cdot \restriction_{\mathbb Z^2-\Lambda_{n,m}})$, so that in general, the connections interior to $\Lambda_{n,m}$ do not count towards the induced boundary conditions. Define the \emph{free-at-infinity} boundary conditions analogously as the distribution given by $\pi_{\mathbb Z^2}^{0}(\cdot \restriction_{\mathbb Z^2 - \Lambda_{n,m}})$.

We say that a distribution $\mathbf P$ on boundary conditions on $\Delta$ dominates $\mathbf P'$ if $\xi'\sim \mathbf P'$ is a stochastically finer partition than $\xi\sim\mathbf P$. Moreover, if boundary conditions are given by different distributions on different subsets of $\partial \Lambda_{n,m}$, then define the overall distribution on boundary conditions by sampling the boundary partitions independently on each of the boundary subsets (with no connections between the subsets). We say such a boundary condition \emph{piecewise} dominates wired-at-infinity and is dominated by free-at-infinity.
\end{definition}

A particular example of random boundary conditions that we will be useful later on  is the following: if $R\subset D$ are two concentric rectangles, then the boundary conditions induced by the connections from $\omega\restriction_{D-R}$, where $\omega$ is sampled from $\pi^{1}_{R}$, dominate the wired-at-infinity boundary conditions. 

\subsection{A canonical paths estimate for 2D random cluster models}

The canonical paths technique has yielded (see~\cite{DiSa93,DiSt91,JeSi89,Sinclair92}) a very useful upper bound on the mixing time of spin-systems on general graphs; namely, the mixing time is at most exponential in the cut-width of the underlying graph (so that on rectangular subsets of $\mathbb Z^2$, it is at most exponential in the shorter side length). In the case of the random cluster model, the long-range interactions complicate this for general boundary condition, but we prove a modified version of such an estimate for a wide class of boundary conditions. 

\begin{proposition}\label{prop:canonical-paths}
Let $q>4$, $p= p_c(q)$, and $\phi\in (-\frac{\pi}2, \frac{\pi}2)$. Consider the FK model on $S=\mathcal S_{b,h,\phi}\cap \Lambda_{n,n}$ with boundary conditions $\xi \sim \mathbf P$, where $\mathbf P$ is arbitrary on $\partial S \cap \partial_{\east,\west} \Lambda_{n,n}$, and piecewise dominates wired-at-infinity or is dominated by free-at-infinity on each of the other sides of $\partial S$. There exists $C_q>0$ such that for every sequence $f_n\to \infty$, the FK Glauber dynamics have 
\begin{align*}
\mathbf P\big(\xi: \tmix^{\xi} \geq |E(S)|^2\exp [2 (4h + f_n) \log q]\big)\leq O(n^2 e^{-C_q f_n})\,.
\end{align*}
\end{proposition}

The proof will use the following straightforward comparison estimate.

\begin{lemma}\label{lem:bridges-cutwidth}
Fix any $p\in (0,1)$ and $q\geq 1$ and consider the FK model on $S=\mathcal S_{b,h,\phi} \cap \Lambda_{n,n}$ with arbitrary boundary conditions $\xi$. For a connected edge subset $F\subset E(S)$, we have for every $\omega\in \{0,1\}^{E(S)}$ and every $\zeta,\zeta' \in \{0,1\}^F$
\begin{align*}
\pi^{\xi}_S(\omega \mid \omega\restriction_F= \zeta) \leq q^{|\partial F-\partial S| + \tilde k(\xi,F)-1} \pi_S^{\xi}(\omega \mid \omega\restriction_F = \zeta')
\end{align*}
where $\tilde k(\xi,F)$ is the number of components of $\xi$ that intersect both $\partial F$ and $\partial S-\partial F$.
\end{lemma}

\begin{proof} 
By the domain Markov property, the difference between $\pi_S^{\xi}(\cdot  \mid \omega\restriction_F = \zeta)$ and $\pi_S^{\xi}(\cdot \mid \omega\restriction_F = \zeta')$ is in the boundary conditions that $\zeta\cup \xi$ and $\zeta'\cup \xi$ induce on $S-F$. The boundary partitions induced can differ arbitrarily on $\partial F-\partial S$ contributing a factor of $q^{|\partial F-\partial S|}$; on the other hand, since both configurations use $\xi$ boundary conditions on $S$, the boundary partitions on the rest of $\partial (S-F)$ can only differ if $\zeta$ induces additional boundary connections between distinct boundary components of $\xi$ that reach $\partial S-\partial F$; this accounts for the factor of $q^{\tilde k(\xi,F)-1}$.   
\end{proof}

\begin{proof}[\textbf{\emph{Proof of Proposition~\ref{prop:canonical-paths}}}]
We modify the proof of the canonical path estimate for spin-systems with short-range interactions to the present setup. Partition every segment $L$ of $\partial S\setminus \partial_{\east,\west} \Lambda_{n,n}$ on which $\xi$ is independently sampled, into $\lfloor\frac{|L|}{f_n}\rfloor$ sub-segments $\ell_i$ of $f_n $ vertices each, and possibly an additional exceptional segment $\ell_0$. For every edge $e\in L$, let $\Gamma^e$ be the set of components of $\xi$ that contain vertices on both sides of $L-\{e\}$, which we hereafter refer to as \emph{bridges}. Let
\begin{align}\label{eq:no-bridge-bc}
\mathcal E_{f_n} = \bigg\{\xi: \bigcap_L \bigcap_{i\geq 1} \{\exists e\in \ell_i \mbox{ s.t. } |\Gamma^e| \leq 1\} \bigg \}\,.
\end{align} 
We first prove that for every $\xi \in f_n$, the FK Glauber dynamics on $S$ has
\begin{align}\label{eq:gap-canonical-paths-bd}
\gap_{\xi,S} \leq |E(S)|(1+\sqrt q)\exp [2(4h+f_n) \log q ]\,. 
\end{align} 
This will follow from a standard application of the canonical paths argument for spin-systems with short-range interactions. Namely, label the edges in $S$ lexicographically in their midpoint, first by horizontal coordinate, then by vertical coordinate, and define the path $\gamma_{\omega,\omega'}$ (identified with a sequence of edges in $\Omega_\rc$ between FK configurations, $\omega$ and $\omega'$) as follows: let $e_{l_1},...,e_{l_k}$ be the sequence of edges on which $\omega(\tilde e_{l_i})\neq \omega'(\tilde e_{l_i})$, labeled in their lexicographic ordering. The $i$'th edge in $\gamma_{\omega,\omega'}$ will then be between configuration 
\begin{align*}
\eta=\omega'\restriction_{\{e_1,...,e_{l_i-1}\}} \cup \omega \restriction_{\{e_{l_i},...,e_{|E(S)|}\}}
\end{align*}
and its neighbor $\eta'$ which also has $\eta'(\tilde e_i) = \omega'(\tilde e_i)$. Also, let $\eta^*$ be the configuration that is instead given by $\eta^* = \omega\restriction_{\{e_1,...,e_{l_i -1}\}} \cup \omega'\restriction_{\{e_{l_i},...,e_{|E(S)|}\}}$.
Then by Lemma~\ref{lem:bridges-cutwidth} applied with the choice of $F$ being $\{e_1,...,e_{l_i - 1}\}$ or $\{e_{l_i},...,e_{|E(S)|}\}$, 
\begin{align*}
\pi_S^{\xi}(\omega)\pi_S^{\xi}(\omega') \leq \pi_S^{\xi}(\eta)\pi_S^{\xi}(\eta^*) q^{2(4h+f_n)}
\end{align*}
as $|\partial F-\partial S| \leq 2h$ and $\tilde k(\xi,F) \leq 2h+f_n$. This follows from the fact that $\xi \in \mathcal E_{f_n}$ and the nested structure of boundary bridges, and the fact that the sides with arbitrary boundary conditions have height at most $2h$. By construction, for every transition $(\eta,\eta')$, the map $(\omega,\omega')\mapsto (\eta,\eta^*)$ is injective. Moreover, the probability of making any transition in $\Omega_\rc$ is bounded below by $\frac 1{1+\sqrt q}$. Putting all this together, by the path method we see that for every $\xi \in \mathcal E_{f_n}$, Eq.~\eqref{eq:gap-canonical-paths-bd} holds and by~\eqref{eq:gap-tmix}, the corresponding bound with an extra factor of $O(|E(S)|)$ also holds for the mixing time.

It remains to bound the $\mathbf P$-probability that $\xi\in \mathcal E_{f_n}$.  Fix a segment $L$ of $\partial S\setminus \partial_{\east,\west}\Lambda_{n,n}$ on which $\mathbf P$ is piecewise sampled, and fix a sub-segment $\ell_i$, then take a union bound over all such segments and all $\ell_i$. Moreover, suppose that the boundary conditions on the segment $L$ are dominated by free-at-infinity (the estimate for the case when the distribution of $\mathbf P$ on $\xi\restriction_{L}$ dominates wired-at-infinity follows by similar reasoning). Since $|\Gamma^e|\geq 1$ is an increasing event, it suffices to show
\begin{align*}
\pi_{\mathbb Z^2}^{0}(\exists e\in \ell_i  \mbox{ s.t. } |\Gamma^e|=0) \geq 1-O(n^2 e^{-C_q f_n})\,.
\end{align*}
However, by planarity of boundary conditions induced by $\pi_{\mathbb Z^2}^{0}$ on $\partial S$, the complement of the left-hand side is the event that there exist $x,y$ in the two parts of $L-\ell_i$ such that $x\longleftrightarrow y$, which, if $C_q$ is the constant from~\eqref{eq:exp-decay}, has probability at most 
\begin{align*}
|L-\ell_i| e^{-C_q  f_n} \leq 2ne^{-C_q  f_n}\,.
\end{align*}
(For the wired-at-infinity boundary conditions, observe that in order for $|\Gamma^e| > 1$ for every $e\in \ell_i$,  there must exist $x,y$ in the two parts of $L-\ell_i$ such that $x\stackrel{\ast} \longleftrightarrow y$: this is in turn a decreasing event with exponentially decaying probability under $\pi^1_{\mathbb Z^2}$.) Taking a union bound over at most $4n$ sub-segments $\ell_i$ of various segments $L$, we obtain that
\[
\mathbf P(\xi\in \mathcal E_C) \leq 8 n^2 e^{-C_q f_n}\,. \qedhere
\]
\end{proof}

\begin{remark}\label{rem:canonical-paths-extension}
By standard comparison estimates, one could allow \emph{arbitrary} boundary conditions on any boundary segment of size $O(h)$ of $\partial S$, paying a cost in the spectral gap of at most $\exp(ch \log q)$. This would follow from bounding the ratio of the Dirichlet forms and the Radon-Nikodym derivative between the two (see e.g.,~\cite[Lemma 2.8]{Martinelli94} and~\cite[Eq.~(5.1)]{GL16a} for details). 
\end{remark}

\subsection{Dobrushin boundary conditions}\label{sub:subexp-wired-free}

In this section, we consider the mixing time of Swendsen--Wang dynamics with boundary conditions that are free on a subset of~$\partial \Lambda$ and red elsewhere. While \S\ref{sec:slow-mixing} demonstrates that such boundary conditions can induce a slow mixing (at least $\exp(c n)$) by respecting the order-disorder phase symmetry, this section will establish that the mixing time is faster (at most $\exp(n^{1/2+o(1)})$) under  boundary conditions that have a single order-disorder interface.

Define a general class of \emph{order-disorder Dobrushin boundary conditions}, whose FK representation is wired on one connected boundary arc and free elsewhere. Let $a_n,b_n$ be two distinct points on $\partial \Lambda_{n,n}$. For marked boundary points $(a_n,b_n)$, FK Dobrushin boundary conditions are those that are wired on the clockwise (starting from the origin) arc $(a_n,b_n)$ and free on $(b_n,a_n)$.

\subsection*{Sketch of proof}
Our proof of Theorem~\ref{mainthm:2} adapts the proof of $\tmix \lesssim \exp(c\sqrt{n\log n})$ in~\cite{Martinelli94} for the low-temperature Ising model with plus boundary conditions, but using the censoring inequality Theorem~\ref{thm:censoring} instead of the block dynamics approach of~\cite{Martinelli94}. For Dobrushin boundary conditions between $(a,b)$, we sequentially (at times $t_i$) censor all updates except those in the strip $B_i$ of height $\sqrt{n\log n}$ parallel to the line segment $\langle a,b\rangle$ such that $B_i$ and $B_{i+1}$ overlap on half their height. On the one hand, the canonical paths estimate Proposition~\ref{prop:canonical-paths} bounds the mixing time of $B_i$ by exponential in its height, so that we take $t_i - t_{i-1} =O(e^{c\sqrt{n\log n}})$. On the other hand, by Theorem~\ref{thm:censoring}, if the censored chains started from all wired and all free are coupled with high probability, that bounds the mixing time of the original chain. To couple these two chains, we systematically push the interfaces of the chains started from these initial configurations down (resp., up), until they are within $O(\sqrt{n\log n})$ of each other. This is possible because with high probability (see Lemma~\ref{lem:subset-surface-tension}) the interface of the bottom boundary of $B_i$ (where the free initial configuration is seen) never reaches the top of $B_{i-1}$ and the censored dynamics continues pushing the interface down to $O(\sqrt{n\log n})$ distance of $\langle a,b\rangle$ (see Figure~\ref{fig:pushing-the-interface}).

\begin{proof}[\textbf{\emph{Proof of Theorem~\ref{mainthm:2}}}]
If $a_n,b_n$ are on the same side of $\partial \Lambda_{n,n}$, rotate $\Lambda_{n,n}$ so that they are both on $\partial_\south \Lambda_{n,n}$ and the angle of the interface between them will be zero. Then the proof below when $a_n$ and $b_n$ are on different sides applies identically; the only difference is the boundary conditions in the last step of recursion, where the identity coupling still couples all FK chains with probability $1$ in the mixing time of that last block.

Now suppose $a_n= (a_n^1,a_n^2)$ and $b_n= (b_n^1,b_n^2)$ are on different sides of $\partial \Lambda_{n,n}$; by rotational and reflective symmetry and self-duality, we can take $a_n\in \partial_\west \Lambda_{n,n}$ to be the first point encountered clockwise from the origin, and ensure that $\phi_n= \tan^{-1} (\frac {a_n^2 - b_n^2}{a_n^1-b_n^1})$ is such that $\phi_n \in [-\frac \pi 4,\frac \pi 4]$. Fix any such choice of $a_n,b_n\in \partial \Lambda_{n,n}$ and let $\phi=\phi_n$.

We establish the theorem for FK Glauber dynamics with $(a_n,b_n)$ Dobrushin boundary conditions. Throughout this proof, let $c_1=c_1(q)>0$ be a  large enough constant (e.g., $5/C_q$ for $C_q$ from Proposition~\ref{prop:canonical-paths} would suffice).
Define the overlapping blocks
\begin{align*}
B_i &:= \cS_{{a_n}+{(N-i)\ell},\ell,\phi}\cap \Lambda_{n,n} \qquad(i=1,\ldots , N)\qquad\mbox{ and}\\
B_i &:= \cS_{{a_n}-{(N-i)\ell},\ell,\phi}\cap \Lambda_{n,n} \qquad(i=-1,\ldots , -N+1)\qquad\mbox{for $N=\big\lceil\tfrac {n}{\ell}\big\rceil$}\,,
\end{align*}
where we choose $\ell=c_3\sqrt {n\log n}$ for $c_3=4/\sqrt{c_2}$, with $c_2(q)$ as given by Proposition~\ref{prop:surface-tension},
so $N\sim c_3^{-1}\sqrt{n/\log n}$ (see Figure~\ref{fig:regions}).
Because of our choice of $(B_{\pm j})_{j=1}^N$, as many as $N$ of the $B_j$ may be empty, and henceforth if $B_j$ is empty, we say any associated events hold trivially. 
Let 
\begin{align*}
t_i = i \cdot K \log n \cdot N^2 \cdot |E(B_i\cup B_{-i})|^2 e^{4(4\ell + c_1\log n) \log q} \qquad \mbox{for} \qquad 0\leq i \leq N\,,
\end{align*} 
for $K$ to be chosen large later.
Define the censored chain $\bar X_t$: between times $t_{i-1}$ and $t_{i}$ censor all updates except those in $B_{i}\cup B_{-i}$. Let $\bar X_t^1$ be the censored chain started from $\bar X_0 =1$ and $\bar X_t^0$ be the censored chain started from $\bar X_0= 0$. By Theorem~\ref{thm:censoring} and~\eqref{eq:init-config-comparison}, it suffices to show that there exists a coupling of $\bar X_t^0$ and $\bar X_t^1$ such that
\begin{align}
\mathbb P(\bar X_{t_{N}}^1\neq \bar X_{t_{N}}^0) = o(n^{-2})\,,
\end{align}
since for any $K, c_1$ fixed, we have $t_{N} \lesssim \exp(O(\sqrt{n\log n}))$ as desired.

\begin{figure}
  \begin{tikzpicture}
    \node (plot1) at (0,0) {};

    \begin{scope}[shift={(plot1.south west)},x={(plot1.south east)},y={(plot1.north west)}, font=\small]

\draw[color=black, dashed] (0,0) rectangle (20,20);
\draw[draw=black, ultra thick] (0,10) -- (0,20)--(20,20)--(20,0);

\draw[draw=black, fill=green, opacity=.15] (0,16.5)--(20,6.5)--(20,9.5)--(0,19.5);
\draw[draw=black, fill=blue, opacity=.15] (0,15)--(20,5)--(20,8)--(0,18);
\draw[draw=black, fill=gray, opacity=.35] (0,18)--(0,20)--(20,20)--(20,8)--(0,18);

\draw[draw=black, opacity=.25] (0,16.5)--(20,6.5)--(20,9.5)--(0,19.5);

\draw[draw=none, pattern=vertical lines, pattern color=DarkGreen](0,16.5)--(20,6.5)--(20,9.5)--(0,19.5);
\draw[draw=none, pattern=horizontal lines, pattern color=blue, opacity=.75] (0,15)--(20,5)--(20,8)--(0,18);
\draw[draw=none, pattern = crosshatch dots, opacity=.2] (0,18)--(0,20)--(20,20)--(20,8)--(0,18);

\draw[color=gray, thin] (0,8.5)--(0,11.5)--(20,1.5)--(20,0)--(17,0)--(0,8.5);

    \node[font=\small] at (-1,10) {$a_n$};  
        \node[font=\small] at (21,0) {$b_n$};  
    \node[font=\small] at (27.3,17.5) {$: B_{i}$};
        \node[font=\small] at (28,15.5) {$: B_{i+1}$};
            \node[font=\small] at (27.5,13.5) {$: R_i^+$};

    \draw[draw=black, fill=green, opacity=.25] (24,18)--(25,18)--(25,17)--(24,17)--(24,18);
    \draw[draw=none,pattern=vertical lines, pattern color=DarkGreen] (24,18)--(25,18)--(25,17)--(24,17)--(24,18);    
    \draw[draw=black, fill=blue, opacity=.25] (24,16)--(25,16)--(25,15)--(24,15)--(24,16);
    \draw[draw=none,pattern=horizontal lines, pattern color=blue, opacity=0.75] (24,16)--(25,16)--(25,15)--(24,15)--(24,16);

    \draw[draw=black, fill=gray, opacity=.35] (24,14)--(25,14)--(25,13)--(24,13)--(24,14);
    \draw[draw=none, pattern = crosshatch dots, opacity=.2] (24,14)--(25,14)--(25,13)--(24,13)--(24,14);

     \draw[color=gray, opacity=.35] (0,10)--(20,0);
    
        \node[font=\small] at (10,5) {$B_N$};  
      \draw[<-] (20.7,6.5)--(20.7,7.6); \node[font=\tiny] at (20.7,8) {$2\ell$}; \draw[->] (20.7, 8.4)--(20.7,9.5);

      \draw[<-] (-.7,15)--(-.7,16.1); \node[font=\tiny] at (-.7,16.5) {$2\ell$}; \draw[->] (-.7, 16.9)--(-.7,18);
        

     \node[font=\small] at (0,10) {$\bullet$};
     \node[font= \small] at (20,0) {$\bullet$};

    \end{scope}
\end{tikzpicture}
\caption{The blocks $B_{i}$ in $\Lambda$, each of which are wired on all sides except possibly $\partial_\south B_{i}$ and the region $R_i^+$ on which we have coupled.} 
\label{fig:regions}
\end{figure}
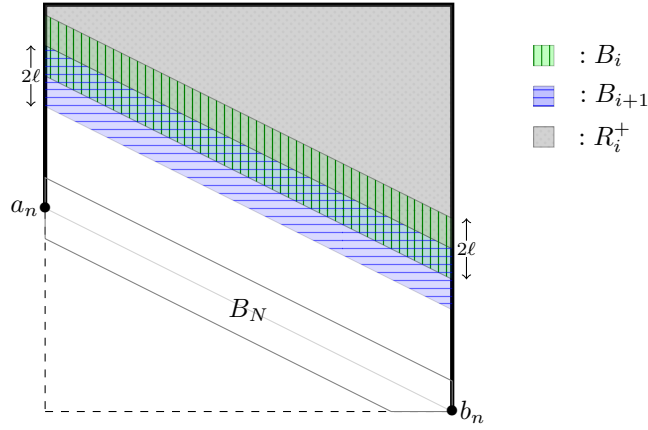

We now define a monotone coupling of $\bar X_{t}^1$ and $\bar X_{t}^0$ which satisfies the above. For each $i=1,\ldots,N$, define the event
\[
A_i^{\pm}= \left\{ \bar X_{t_{i}}^0\restriction_{E(R_i^{\pm})} \neq  \bar X_{t_{i}}^1\restriction_{E(R_i^{\pm})} \right\}\,,\qquad\mbox{where}\qquad
R^{\pm}_i=\bigcup_{j=1}^{i} B_{\pm j}\setminus B_{\pm (i+1)}\,.
\]

We can then write under our coupling, 
\begin{align*}
\P(\bar X_{t_{N}}^0 \neq \bar X_{t_{N}}^1) & \leq \P\bigg(\bigcup_{i=1}^{N} (A_i^+ \cup A_i^-)\bigg) \\
& \leq \P(A_{N-1}^+) + \P (A_{N-1}^-)+\P(A_N\mid (A_{N-1}^{+})^c,(A_{N-1}^-)^c)\,.
\end{align*}
The bounds on $\P(A_{N-1}^+)$ and $\P(A_{N-1}^-)$ are analogous (using the duality of the FK model) and therefore we only bound the former. Abusing notation slightly, when we consider the restriction of the chain $\bar X_t^{1/0}$ to a boundary $\partial S$ , we mean the boundary conditions induced on that line by $\bar X_{t}^{1/0}\restriction_{\Lambda-S}$. 
We will prove the following inductively.
\begin{claim}\label{claim:induction1}
There exists $c(q)>3$ so that, for large enough $K, c_1$, the following holds.
\begin{enumerate}[(1)]
	\item For every $1\leq i \leq N-1$,
	there exists an event $F_i$ measurable w.r.t.\ $(\bar X_t^0,\bar X_t^1)_{t\leq t_i}$ such that $\mathbb P(F_i) \geq 1-O(in^{-c})$ and $\mathbb P(\bar X_{t_{i}}^{0} \restriction_{\partial_\north B_{i+1}} \in \cdot \mid F_i)$ dominates the boundary conditions induced by $\pi_{R_{N-1}^+}^{1}(\cdot \restriction_{R_{i}^+})$ on $\partial_\north B_{i+1}$.
	\item For every $1\leq i \leq N-1$, \begin{align*}
\mathbb P(A_{i}^+)= \mathbb P(\bar X_{t_{i}}^1\restriction_{R_i^+} \neq \bar X_{t_{i}}^0\restriction_{R_i^+}) \lesssim i n^{-c} \,.
\end{align*}
\end{enumerate} 
\end{claim}
\begin{proof}[\textbf{\emph{Proof of Claim~\ref{claim:induction1}}}]For the base case, (1) holds trivially because the outer boundary on $B_1$ is just given by the boundary of $\partial \Lambda$ which will be all-wired there. The base case proof of (2) is just a simplification of the proof of the inductive step for (2) so we do not repeat it here. 
Now suppose both (1) and (2) hold for some $i-1$ and prove they hold for $i$ for a $c(q)>0$ we will pick later. By the inductive hypothesis, with probability $1-O((i-1)n^{-c})$, we have $\bar X_{t_{i}}^0 \restriction_{R_{i-1}^+} = \bar X_{t_{i}}^1 \restriction_{R_{i-1}^+}$, and the boundary conditions on $\partial_\north B_{i}$ dominate wired-at-infinity. In particular, since on the other sides of $B_{i}$ both chains' boundary conditions are either all-free or all-wired, by Proposition~\ref{prop:canonical-paths}, with probability $1-O((i-1)n^{-c})-O(n^{-C_q c_1+2})$ (where $C_q>0$ is the constant from that proposition) the mixing time of FK Glauber dynamics with boundary conditions given by $\bar X_{t_{i}}^1\restriction_{\partial B_{i}}$ is at most $|E(B_i\cup B_{-i})|^2 e^{4(4\ell+ c_1\log n) \log q}$. Now suppose that both of these events hold and consider the probability that $\bar X_{t_{i}}^1\restriction_{R_i^+} \neq \bar X_{t_{i}}^0\restriction_{R_i^+}$.

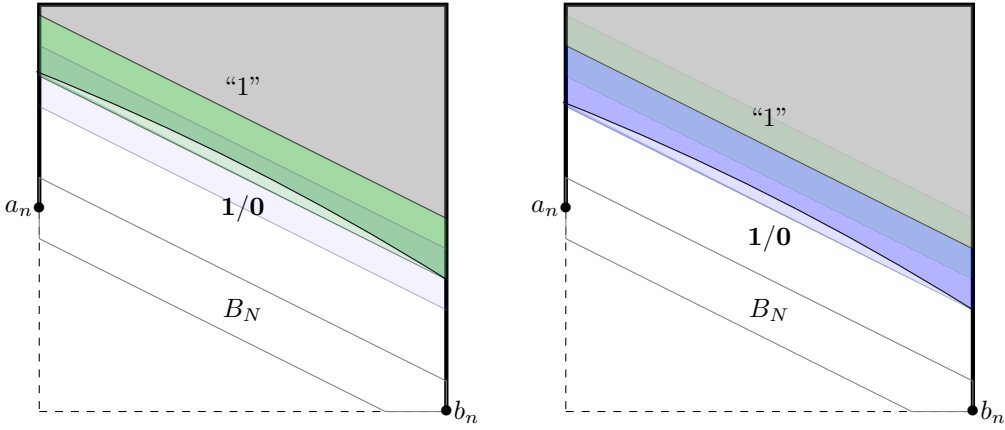
\begin{figure}
  \begin{tikzpicture}
    \node (plot1) at (7,0) {};

    \node (plot2) at (0,0) {};

    \begin{scope}[shift={(plot1.south west)},x={(plot1.south east)},y={(plot1.north west)}, font=\small]

\draw[color=black, dashed] (0,0) rectangle (20,20);
\draw[draw=black, ultra thick] (0,10) -- (0,20)--(20,20)--(20,0);
\draw[draw=black, fill=green, opacity=.1] (0,16.5)--(20,6.5)--(20,9.5)--(0,19.5);
\draw[draw=black, fill=blue, opacity=.3] (0,15)--(20,5)--(20,8)--(0,18);
\draw[draw=black, fill=gray, opacity=.4] (0,18)--(0,20)--(20,20)--(20,8)--(0,18);

\draw[ color=black] (20,5) arc (58:68.2:127) ;
\draw[ color=black, fill=white, opacity=.6] (20,5) arc (58:68.2:127);

\draw[color=gray, thin] (0,8.5)--(0,11.5)--(20,1.5)--(20,0)--(17,0)--(0,8.5);

    \node[font=\small] at (-1,10) {$a_n$};  
        \node[font=\small] at (21,0) {$b_n$};  
    \node[font=\small] at (10,8.5) {$\mathbf{1}/\mathbf{0}$};  
    
        \node[font=\small] at (10,5) {$B_N$};  
        
       \node[font=\small] at (10,14.5) {``$1$"};

     \node[font=\small] at (0,10) {$\bullet$};
     \node[font= \small] at (20,0) {$\bullet$};

    \end{scope}

    \begin{scope}[shift={(plot2.south west)},x={(plot2.south east)},y={(plot2.north west)}, font=\small]

\draw[color=black, dashed] (0,0) rectangle (20,20);
\draw[draw=black, ultra thick] (0,10) -- (0,20)--(20,20)--(20,0);
\draw[draw=black, fill=green, opacity=.3] (0,16.5)--(20,6.5)--(20,9.5)--(0,19.5);
\draw[draw=black, fill=gray, opacity=.3] (0,16.5)--(20,6.5)--(20,9.5)--(0,19.5);

\draw[draw=black, fill=gray, opacity=.4] (0,19.5)--(0,20)--(20,20)--(20,9.5)--(0,19.5);
\draw[draw=gray, thin, opacity=.5] (0,18)--(20,8);

\draw[draw=gray, thin, opacity=.5] (0,15)--(20,5);
\draw[fill=blue, opacity=.05] (0,15)--(0,18)--(20,8)--(20,5)--(0,15);

\draw[ color=black] (20,6.5) arc (58:68.2:127) ;
\draw[ color=black, fill=white, opacity=.6] (20,6.5) arc (58:68.2:127) ;

\draw[color=gray, thin] (0,8.5)--(0,11.5)--(20,1.5)--(20,0)--(17,0)--(0,8.5);

     \node[font=\small] at (0,10) {$\bullet$};
     \node[font= \small] at (20,0) {$\bullet$};

    \node[font=\small] at (-1,10) {$a_n$};  
        \node[font=\small] at (21,0) {$b_n$};  
    \node[font=\small] at (10,10) {$\mathbf{1}/\mathbf{0}$};  
    
        \node[font=\small] at (10,5) {$B_N$};  
        
        \node[font=\small] at (10,16) {``$1$"};
    \end{scope}
  \end{tikzpicture}
  \caption{The green and blue blocks $B_{k}, B_{k+1}$ are updated by the censored dynamics in two consecutive steps; with high probability, the chain $(\bar X_{t_k}^0)_{k}$ pushes its interface down toward $\langle a,b\rangle$ by $\ell$ at every step, and is subsequently coupled to $(\bar X_{t_k}^1)_k$ on the growing gray region.} 
  \label{fig:pushing-the-interface}
\end{figure}

By submultiplicativity of total-variation distance, with probability $1-O(n^{-K})$ there exists a coupling of $\bar X_{t_{i}}^{1/0}\restriction_{B_{i}}$ to $\pi_{B_{i}}^{\bar X_{t_{i-1}}^{1/0}}$. It remains to compute the probability of (1) obtaining boundary conditions under $\pi_{B_{i}}^{\bar X_{t_{i-1}}^{0}}$ that dominate those induced by $\pi_{R_{N-1}^+}^1(\cdot \restriction_{R_{i}^+})$ on $\partial_\north B_{i+1}$, and (2) succeeding in coupling $\pi_{B_{i}}^{\bar X_{t_{i-1}}^{1}}$ to $\pi_{B_{i}}^{\bar X_{t_{i-1}}^{0}}$ on $B_{i}-B_{i+1}$ (and therefore on all of $R_i^+$).

Under the above events, let $(\zeta,0)$ be the boundary conditions induced on $\partial B_{i}$ by $\bar X_{t_{i-1}}^0$ and $(\zeta,1)$ be those induced by $\bar X_{t_{i-1}}^1$ where $\zeta$ is a random boundary condition sampled from a distribution dominating $\pi_{R_{N-1}^+}^1 (\cdot \restriction_{R_{i-1}^+})$ on $\partial_\north B_{i}$. Then, the monotone coupling of $\pi^{\zeta,0}_{B_{i}}$ to $\pi^{\zeta,1}_{B_{i}}$ couples the two on $B_{i}-B_{i+1}$ whenever the bottom-most horizontal crossing $\mathcal I$ in the sample from $\pi_{B_{i}}^{\zeta,0}$ has $\mathcal I \subset B_{i+1}$ (see Figure~\ref{fig:pushing-the-interface}). In that case, by revealing dual-edges from the bottom up, the configurations from $\pi^{\zeta,0}_{B_{i}}$ and $\pi^{\zeta,1}_{B_{i}}$ could be coupled above that interface and in particular on all of $B_{i}-B_{i+1}$.  
(Observe that conditioning on the configuration below the interface, in order to reveal $\mathcal I$, cannot affect the boundary conditions above it because on each side of $\partial (R_i^+\cup B_{i+1})$ the boundary conditions are all-wired or all-free and additional connections cannot be induced (cf.\ the boundary bridges of \cite{GL16b}).)

Observe, also, that when considering $\zeta$ dominating the boundary conditions induced by $\pi_{R_{N-1}^+}^1(\cdot \restriction_{R_{i-1}^+})$, since the boundary on $\partial B_{i+1}\cap \partial \Lambda$ is wired, by the domain Markov property, the boundary conditions induced by $\bar X_{t_{i+1}}^0$ on $\partial_\north B_{i+1}$ when $\mathcal I \subset B_{i+1}$ holds will dominate $\pi_{R_{N-1}^+}^1(\cdot \restriction_{R_{i}^+})$. This implies part (1) if we can bound $\pi^{\zeta,0}_{\partial_\north B_{i}}(\mathcal I \subset B_{i+1})$. 

Thus, for both (1) and (2) of the induction, it only remains to bound the probability
\begin{align}\label{eq:chain-interface-fluctuation}
\mathbb E\bigg[\pi_{B_{i}}^{\zeta,0}\big(\mathcal I \not \subset B_{i+1}\big)\given F_{i-1}\bigg] & \leq \mathbb E_{\pi_{R_{N-1}^+}^1} \bigg[\pi_{B_{i}}^{\omega\restriction_{R_{i-1}^+},0}(\mathcal I \not \subset B_{i+1})\bigg] \nonumber \\ 
&\leq \pi^{1,0}_{R_i^+\cup B_{i}}(\mathcal I \not \subset B_{i+1})\,,
\end{align}
where ${(1,0)}$ boundary conditions on $R_{i}^+\cup B_{i}$ are free on $\partial_\south B_{i}$ and wired elsewhere.  

In that case, Lemma~\ref{lem:subset-surface-tension} (noting that the estimate there was independent of $b$) with the choice of $a=\frac12 c_3 \sqrt {n \log n}$ implies that the probability in~\eqref{eq:chain-interface-fluctuation} is at most $O(n^{-6})$. Combining all of the above, the probability that items (1) and (2) hold is at least 
\begin{align*}
1-O((i-1)n^{-c})+O(n^{-C_q c_1+2})+O(n^{-K})+O(n^{-6})\,,
\end{align*}
which concludes the proof of the induction as long as we take $c_1, K$ large enough that that the latter three terms are all $o(n^{-c})$. \end{proof}

By Claim~\ref{claim:induction1}, we see that both $\mathbb P(A_{N-1}^+)$ and $\mathbb P(A_{N-1}^-)$ have probability at most $O(N n^{-3})$ which is $o(n^{-2})$. It remains to bound $\mathbb P(A_N \mid (A_{N-1}^+)^c, (A_{N-1}^-)^c)$ using similar reasoning to the above. First of all, by part (2), with probability $1-o(n^{-2})$ the chains $\bar X_{t_{N-1}}^{0}$ and $\bar X_{t_{N-1}}^1$ are coupled on both $\partial_\north B_N$ and $\partial_\south B_N$. Moreover, by part (1), the boundary conditions they induce dominates wired-at-infinity on $\partial_{\north}B_N$ with probability at least $1-o(n^{-2})$ and likewise, are dominated by free-at-infinity on $\partial_\south B_N$ with similar probability. Therefore, by time $t_{N}$, by submultiplicativity of total-variation distance and Proposition~\ref{prop:canonical-paths}, we have 
\begin{align*}
\mathbb P(A_{N} \mid (A_{N-1}^+)^c,(A_{N-1}^-)^c) \leq 1- o(n^{-2}) -O(n^{-C_q c_1 + 2})-O(n^{-K})\,,
\end{align*}
and for $c_1,K$ large enough, the right-hand side is $1-o(n^{-2})$.
\end{proof}

\subsection{Sub-exponential mixing on cylinders}\label{sub:subexp-cylinders}

For the rectangle $\Lambda_{n,n}$ define boundary conditions $(p,R)$ (resp.\ $(p,0)$ or $(p,R,0)$ boundary conditions) to be periodic boundary conditions on $\partial_{\north,\south} \Lambda_{n,n}$ and red boundary conditions on $\partial_{\east,\west} \Lambda_{n,n}$ (resp.\ free on $\partial_{\east,\west}\Lambda_{n,n}$ or red on $\partial_{\west} \Lambda_{n,n}$ and free on $\partial_\east \Lambda_{n,n}$). 
We prove the mixing time upper bounds on cylinders with the above boundary conditions (Theorem~\ref{mainthm:3}) at the same time. In what follows, we use $c>0$ to denote the existence of a constant (possibly depending on $q$), where different appearances of $c$ at different places may refer to different values.

The proof builds on the proof of Theorem~\ref{mainthm:2} in that we use the censoring inequalities to push the FK order-disorder across $\Lambda_{n,n}$ in order to couple the chains $\bar X_t^1$ and $\bar X_t^0$. We consider the censored dynamics that sequentially update $N=O(n^{\frac 12-\epsilon})$ overlapping vertical strips of width $n^{\frac 12 +\epsilon}$, ordered from left to right. However, unlike the case in Theorem~\ref{mainthm:2}, our strips do not have wired boundary conditions on three sides (their boundary conditions on the top and bottom are periodic), and therefore, the interface is pushed to the next strip to be updated with probability $\exp({-cn^{2\epsilon}})$ (rather than $1-o(1)$): see Figure~\ref{fig:pushing-the-interface-2}. Thus, with probability $\exp({-cn^{\frac 12+\epsilon}})$  the interface moves to next strip in $N$ consecutive time steps, so that one will succeed, with high probability, at pushing the interface completely across $\Lambda_{n,n}$ after $\exp(n^{\frac 12+2\epsilon})$ attempts. 

\begin{proof}[\textbf{\emph{Proof of Theorem~\ref{mainthm:3}}}]
We again prove the upper bound for the FK Glauber dynamics which translates to an upper bound on Swendsen--Wang dynamics by Theorem~\ref{thm:Ullrich-comparison}. Fix any $\epsilon,\delta>0$ small and consider blocks $B_i$ for $i=1,...,N$ where $N=\lceil \tfrac n \ell \rceil -1$, given by
\begin{align*}
\left\{
\begin{array}{l}
{B_{2i-1}  = \llb (i-1)\ell, (i+1)\ell\rrb \times\llb \delta n, (1-\delta) n\rrb} \\
{B_{2i}= \llb (i-1) \ell, (i+1)\ell\rrb \times \llb 0,\lfloor \frac n2\rfloor- \delta n  \rrb \cup \llb \lfloor \frac n2\rfloor +\delta n, n\rrb}
\end{array}
\right\}
\quad \mbox{where} \quad \ell= n^{\frac 12 +\epsilon}\,,
\end{align*}
and $B_{2N+1}=B_0= \Lambda - \cup _{i=1}^{N} B_{2i-1}\cup B_{2i}$. Since the boundary conditions on $\partial_{\north,\south}\Lambda_{n,n}$ are periodic, each $B_{2i}$ can be viewed as a single connected rectangle with boundary 
\begin{align*}
\partial_{\north,\south} B_{2i}= \llb (i-1)\ell,(i+1)\ell\rrb \times\{ \lfloor\tfrac n2\rfloor+ \delta n,\lfloor \tfrac n2\rfloor-\delta n\}\,.
\end{align*}
We prove the mixing time upper bound for $(p,1,0)$ boundary conditions.  We will pause to comment where the $(p,1)$ boundary conditions would behave differently (namely only when updating block $B_{2N+1}$), and on why this does not affect the proof. The $(p,0)$ boundary conditions can be treated by the dual version of the argument we present.

We will cycle through the blocks $B_i$ periodically, so define $$B_j = B_{j \bmod (2N+1)}\qquad \mbox{for all $j>2N$}\,.$$
Define the following censored Markov chain $\bar X_t^{1}$ (resp., $\bar X_t^0$) started from initial configuration $1$ (resp., $0$): for all $i\geq 0$,
Let $f_n = n^{\frac 12 +3\epsilon}$ and let
\begin{align*}
t_i  =i \cdot n \cdot N^2 \cdot |E(B_i)|^2 e^{2(4\ell+ f_n) \log q}\,;
\end{align*}  
during times $[t_{i-1},t_{i})$, censor all updates outside block $B_{i}$.
Let 
\begin{equation}\label{eq-T-def}T:= t_{2N+1} \exp( n^{\frac 12+2\epsilon})=\exp(O(n^{\frac 12 +3\epsilon}))\,;
\end{equation}
 by Theorem~\ref{thm:censoring} and~\eqref{eq:init-config-comparison}, it will suffice to show that
\begin{align*}
\mathbb P(\bar X_{T}^1 \neq \bar X_{T}^0 ) = o(n^{-2})\,.
\end{align*}
We begin with a uniform upper bound on the mixing times of $B_i$. 

\begin{claim}\label{claim:uniform-mixing-time-bound}
Let $m\leq (2N+1)\exp(n^{\frac 12 +2\epsilon})$ and, for every $i\leq m$, define the event
\begin{equation}
	\label{eq:unif-mixing-time-event}
	\Upsilon_i = \Big\{ \Big(\tmix^{\bar X_{t_{i-1}}^{0}} (B_{i})\vee \tmix^{\bar X_{t_{i-1}}^{1}} (B_{i})\Big)\leq |E(B_i)|^2 e^{2(4\ell +f_n )\log q}\Big\} \,,
\end{equation}
where the superscript ${\bar X_{t_{i-1}}^{\omega_0}}$ denotes boundary conditions induced by $\bar X_{t_{i-1}}^{\omega_0}$ on $B_i$.
Then
\begin{align}\label{eq:induction-uniform-mixing-time}
\mathbb P \Big(\bigcup_{i\leq m} \Upsilon^c_i\Big)\lesssim m ( n^2  e^{-C_q f_n}+e^{-n/2})\,,
\end{align}
where $C_q>0$ is the constant given by Proposition~\ref{prop:canonical-paths}. 
\end{claim} 
\begin{proof}[\textbf{\emph{Proof of Claim~\ref{claim:uniform-mixing-time-bound}}}]
Let $\Xi_i$ be the event that the law of the boundary conditions on $B_i$ under $\bar X_{t_{i-1}}^{0}$ piecewise dominate/are dominated by wired/free-at-infinity, resp., and likewise for $\bar X_{t_{i-1}}^{1}$. We prove inductively that for every $m\leq (2N+1)\exp(n^{\frac 12 +2\epsilon})$, 
\begin{align}\label{eq:double-induction-mixing-time}
\mathbb P \bigg(\bigcup_{i\leq m} (\Upsilon_i^c \cup \Xi_i^c)\bigg) \lesssim m(n^2 e^{-C_q f_n}+e^{-n/2})\,.
\end{align}
The base case, $m=1$, has boundary conditions that are wired on $\partial_\west B_1$ and free/wired on $\partial_{\north,\south,\east} B_1$, and thus $\Upsilon_1 \cap \Xi_1$ holds with probability $1$ by a canonical paths estimate (there are no distinct bridges). Suppose now that~\eqref{eq:double-induction-mixing-time} holds for some $m$; to show that it holds for $m+1$, it suffices to show that the boundary conditions induced by~$\bar X_{t_{m}}^0$, i.e., the chain $\bar X_t^0$ (the bound for the chain $\bar X_t^1$ follows symmetrically).

Assume that the event $\bigcap_{i\leq m}(\Upsilon_i\cap \Xi_i)$ holds. First of all, we notice that for any $i$ satisfying $\Upsilon_i$, by the sub-multiplicativity of total-variation distance and definition of~$t_i$, 
\begin{align}\label{eq:coupling-probability}
\|\mathbb P (\bar X_{t_i}^0\restriction_{B_i} \in \cdot) - \pi^{\bar X_{t_{i-1}}^{0}}_{B_i}\|_\tv \leq e^{-n}\,;
\end{align} 
thus, a union bound over all $i\leq m = O(N \exp(n^{\frac 12+2\epsilon}))$ implies that we may construct a coupling of $(\bar X_{t}^0)$ and some random variables $\bar Z_1,\ldots,\bar Z_m$ such that 
\[ \P\bigg( \bigcap_{i\leq m}\left\{\bar X_{t_i}^{0}\restriction_{B_i} = \bar Z_i\right\}\bigg) \geq 1-e^{-n/2}\quad \mbox{ where $\bar Z_i\sim \pi_{B_i}^{\bar X_{t_{i-1}}^0}$ for each $i$}\,.\]
 We now claim that the boundary conditions induced by $\bar X_{t_m}^{0}$ on $B_{m+1}$ are such that they piecewise dominate/are dominated by wired/free-at-infinity on $\partial B_{m+1}$. Consider the case where $m$ is even (the case $m$ is odd follows analogously). According to~$\Xi_{m-1}$, if we denote by $\zeta$ the boundary conditions induced by $\bar X_{t_{m-2}}^0$ on $\partial B_{m-1}$, then $\zeta$ piecewise dominates/is dominated by wired/free-at-infinity. Hence, when sampling from $\pi_{B_{m-1}}^\zeta$, there would be a well-defined FK order-disorder interface $\mathcal I$ between the boundary subsets that are alternately wired and free. For every such interface, by the domain Markov property, the marginal under $\pi_{B_{m-1}}^\zeta$ on each of the connected components of $E(B_m)\setminus \mathcal I$ either dominates wired-at-infinity or is dominated by free-at-infinity. 

As a consequence, the boundary conditions on $\partial_{\north,\south,\west} B_{m}$ are piecewise sampled from distributions dominating/dominated by wired/free-at-infinity, as are the boundary conditions on the vertical bisector of $ (B_{m-1}\setminus B_{m})$, a subset of $\partial_\west B_{m+1}$. Repeating this reasoning for the update on $B_{m}$, we see that the boundary conditions on $\partial_{\north,\south,\west} B_{m+1}$ are all sampled from distributions that piecewise dominate/are dominated by wired/free-at-infinity. Finally, the same is true of $\partial_\east B_{m+1}$ as it is either completely free if $m\leq 2N+1$ or it is similarly sampled from $\pi_{B_{j}}^{\bar X_{t_{j-1}}^0}$ for some $j<m$ which likewise satisfied $\Xi_{j}$. Thus we deduce that, except with probability $e^{-n/2}$, $\Xi_{m+1}$ holds. 
Then by Proposition~\ref{prop:canonical-paths}, the boundary condition on $B_{m+1}$ induced by $\bar X_{t_m}^0$ is such that 
\begin{align*}
\mathbb P (\tmix^{\bar X_{t_{m}}^{0}}(B_{m+1})\geq |E(B_{m+1})|^2 e^{2(4\ell+f_n)\log q}) \leq O(n^2 e^{-C_q f_n})\,.
\end{align*}
A union bound over the above errors concludes the proof.
\end{proof}

Henceforth, we suppose that the event $\Upsilon_i$ holds for all $i\leq (2N+1) \exp ({n^{\frac 12 +2\epsilon}})$, which is  the case with probability $1-\exp(-\Omega(n^{\frac 12 +3\epsilon}))$ since $f_n =n^{\frac 12 +3\epsilon}$. 

Note that every time increment of $t_{2N+1}$ we make an independent attempt at coupling $\bar X_t^1$ to $\bar X_t^0$, albeit with initial configurations induced by the chains at the end of the last sweep, and once the two chains are coupled on all of $\Lambda$, they will remain coupled for all subsequent times. We will show that there exists $c(\delta,q)>0$ such that for every $k$ and every two configurations $\omega^1_k = \bar X_{t_{(k-1)(2N+1)}}^1$ and $\omega^0_k = \bar X_{t_{(k-1)(2N+1)}}^0$,
\begin{align}\label{eq:probability-of-coupling}
\mathbb P(\bar X^0_{t_{k(2N+1)}}=  \bar X^1_{t_{k(2N+1)}}\mid \omega^1_k,\omega^0_k)& \geq \mathbb P(A_{2kN+k-1}^c\mid \omega^1_k,\omega^0_k) \mathbb P(A_{k(2N+1)}^c\mid A_{2kN+k-1}^c,\omega^1_k,\omega^0_k) \nonumber\\
& \gtrsim \exp(-cn^{\frac 12+\epsilon})\,,
\end{align}
where, in analogy to the proof of Theorem~\ref{mainthm:2}, if $R_i = \bigcup_{j\leq i \bmod (2N+1)} B_j$, 
\begin{align*}
A_{2i}& = \left\{X_{t_{2i}}^0 \restriction_{R_{2i}-B_{2i+1}-B_{2i+2}}\neq X_{t_{2i}}^1\restriction_{R_{2i}-B_{2i+1}-B_{2i+2}}\right\}\,.
\end{align*}
Eq.~\eqref{eq:probability-of-coupling} is sufficient because the probability of not coupling $\bar X_T^0$ and $\bar X_T^1$ by time~$T$ (as defined in~\eqref{eq-T-def})  would then be bounded by \[\mathbb P\left(\mbox{Bin}(\exp(n^{\frac 12+2\epsilon}), \exp(-cn^{\frac 12+\epsilon}))=0\right)=o(n^{-2})\,.\]
In order to lower bound the probability in~\eqref{eq:probability-of-coupling}, we will construct a monotone coupling of the two chains; therefore it suffices to consider the wired and free initial configurations; by the Markov property, it also suffices to only consider the first sweep $k=1$.

Recall from~\eqref{eq:no-bridge-bc} that for any rectangle $\mathcal E_{f_n}$ is the set of boundary conditions on $B_i$ such that in every boundary segment of length $f_n$, there is an edge with at most one boundary component containing vertices on both sides of that edge (bridge). 

\begin{claim}\label{claim:induction2}
There exists $c(\delta,q)>0$ such that, for large $K$ and $c_1$, the following holds.
\begin{enumerate}[(1)]
	\item For every $1\leq i \leq N$, there exists an event $F_i$ measurable w.r.t.\ $(\bar X_t^0,\bar X_t^1)_{t\leq t_{2i}}$ such that $\P(\bar X_{t_{2i}}^0 \restriction_{\partial_{\west}(B_{2i+1}\cup B_{2i+2})}\in \cdot \mid F_{i})$ is in $\mathcal E_{f_n}$ and dominates the wired-at-infinity distribution on boundary conditions on $\partial_\west (B_{2i+1}\cup B_{2i+2})$.
	\item For every $1\leq i \leq N$, the above defined event $F_i$ satisfies
\begin{align*}
\mathbb P (F_i, A_{2i}^c) \gtrsim \exp(-cin^{2\epsilon})\,.
\end{align*}
\end{enumerate}
\end{claim}
\begin{proof}[\textbf{\emph{Proof of Claim~\ref{claim:induction2}}}]For the base case of (1), observe that the boundary conditions on $\partial_\west (B_1 \cup B_2)$ dominate wired-at-infinity as they are all-wired; the base case of (2) follows as the proof of the inductive step does, so we do not repeat it here. Now assume that items (1) and (2) hold for $i-1$ (for $c( q)>0$ to be determined later) and show that they hold for $i$ for the same choice of $c$. Consider the middle rectangle
\[D_i=\llb (i-1)\ell,(i+\tfrac 12)\ell\rrb \times \llb \lfloor \tfrac n2\rfloor - \delta n,\lfloor \tfrac n2 \rfloor +\delta n\rrb\,,
\]
and let $\mathcal I$ be the interface bounding the cluster(s) of $\bar X^0_{t_j}\restriction_{\partial_\west B_{j}}$. Define the events
\begin{align*}
\Gamma_{2i} = \left\{\bar X_{t_{2i}}^0 \restriction_{B_{2i}}:\mathcal I \cap (B_{2i} - B_{2i+2}) = \emptyset\right\}\,,
\end{align*}
and
\begin{align*}
\Gamma_{2i-1} = \left\{\bar X_{t_{2i-1}}^0 \restriction_{B_{2i-1}}: \mathcal I \cap D_{i}=\emptyset\right\}\cap \tilde \Gamma_{2i-1}\,,
\end{align*}
where $\tilde \Gamma_{2i-1}$ is the event that $\mathcal I$ is connected to $e_1, e_2$ in $B_{2i-1}- D_i$, where $e_1, e_2$ are a pair of edges in either side of $\partial_\west B_{2i-1}- \partial_\west D_i$, with at most one bridge over them (such a pair of edges exist since by assumption (1), the boundary conditions on $B_{2i-1}$ are in $\mathcal E_{f_n}$). 
It is clear that both $\Gamma_{2i-1}$ and $\Gamma_{2i}$ are increasing events.
Then, we can lower bound 
\begin{align}\label{eq:probability-to-couple-2}
\mathbb P(A_{2i}^c, F_i) \geq \mathbb P(A_{2i-2}^c ,F_{i-1})\mathbb P\Big(A_{2i}^c\cap F_i \cap \Gamma_{2i-1}\cap \Gamma_{2i} \given A_{2i-2}^c , F_{i-1}\Big)\,.
\end{align}

By the inductive hypothesis, $\mathbb P(A_{2i-2}^c, F_{i-1}) \gtrsim \exp(-c(i-1)n^{2\epsilon})$ and from now on work in the probability space conditioned on $F_{i-1} \cap A_{2i-2}^c$.  Under $F_{i-1}\cap A_{2i-2}^c$, the boundary conditions $\bar X^0_{t_{2i-2}} \restriction_{\partial_\west (B_{2i-1}\cup B_{2i})}$ and $\bar X^1_{t_{2i-2}} \restriction_{\partial_\west (B_{2i-1}\cup B_{2i})}$ are coupled and dominate wired-at-infinity (let $\zeta$ denote that random boundary condition). In the following time increment $[t_{2i-2},t_{2i-1})$, only updates on $B_{2i-1}$ are permitted; since we are working on the event $\Upsilon_{2i-1}$, we can couple $\bar X_{t_{2i-1}}^0\restriction_{B_{2i-1}}$ to $\pi^{\zeta,0}_{B_{2i-1}}$ and $\bar X_{t_{2i-1}}^1\restriction_{B_{2i-1}}$ to $\pi^{\zeta,1}_{B_{2i-1}}$ with probability $1-O(e^{-n})$ by~\eqref{eq:coupling-probability}. 

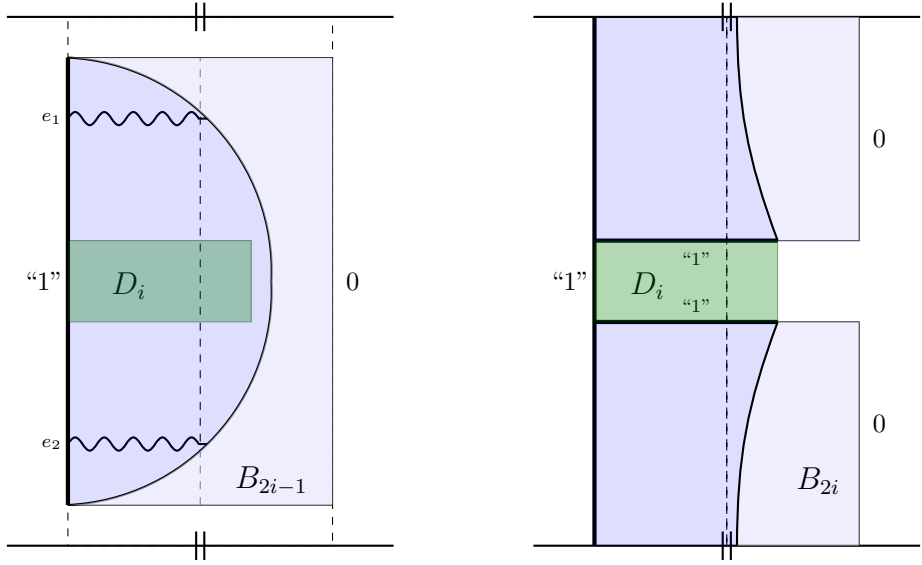
\begin{figure}
  \begin{tikzpicture}
    \node (plot1) at (7,0) {};

    \node (plot2) at (0,0) {};

    \begin{scope}[shift={(plot1.south west)},x={(plot1.south east)},y={(plot1.north west)}, font=\small]

     \draw[draw=black,thick] (-2,-2)--(17,-2);
     \draw[draw=black,thick] (-2,24)--(17,24);
     \draw[draw=black,ultra thick] (1,-2)--(1,24);
     \draw[draw=black,ultra thick] (1,9)--(10,9);
     \draw[draw=black,ultra thick] (1,13)--(10,13);
     
     \draw[draw=black,thick] (7.3,-2.7)--(7.3,-1.3);
     \draw[draw=black,thick] (7.7,-2.7)--(7.7,-1.3);
     \draw[draw=black,thick] (7.3,23.3)--(7.3,24.7);     
     \draw[draw=black,thick] (7.7,23.3)--(7.7,24.7);
     
     \node[font=\small] at (-.2,11) {$``1"$};
     \node[font=\tiny] at (6,12.1) {$``1"$};
     \node[font=\tiny] at (6,9.8) {$``1"$};
     \node[font=\large] at (3.6,10.8) {$D_i$};
     \node[font=\small] at (15,18) {$0$};  
      \node[font=\small] at (15,4) {$0$};

     \draw[draw=black,dashed] (7.5,-2) rectangle (7.5,24);
     \draw[color=black] (1,-2) rectangle (14,9);
     \draw[color=black] (1,13) rectangle (14,24);
     
     \draw[color=black,thin,fill=DarkGreen,opacity=.3] (1,9) rectangle (10,13);
     \draw[color=black, fill=blue, opacity=.13] (1,-2) rectangle (14,9);
     \draw[color=black, fill=blue, opacity=.13] (1,13) rectangle (14,24);

\fill[ color=white, opacity=.5] (8,-2) to [bend right=-10] (10,9)--(14,9)--(14,-2)--(10,-2)--(8,-2);
\draw[ color=black, thick] (8,-2) to [bend right=-10] (10,9) ;

\fill[ color=white, opacity=.5] (8,24) to [bend right=10] (10,13)--(14,13)--(14,24)--(10,24)--(8,24);
\draw[ color=black, thick] (8,24) to [bend right=10] (10,13) ;

     \node[font=\large] at (12,1.2) {$B_{2i}$};

    \end{scope}

    \begin{scope}[shift={(plot2.south west)},x={(plot2.south east)},y={(plot2.north west)}, font=\small]

     \draw[draw=black,thick] (-2,-2)--(17,-2);
     \draw[draw=black,thick] (-2,24)--(17,24);
     \draw[draw=black,thick] (7.3,-2.7)--(7.3,-1.3);
     \draw[draw=black,thick] (7.7,-2.7)--(7.7,-1.3);
     \draw[draw=black,thick] (7.3,23.3)--(7.3,24.7);     
     \draw[draw=black,thick] (7.7,23.3)--(7.7,24.7);
     
     \node[font=\small] at (-.2,11) {$``1"$};  
     
     \node[font=\tiny] at (.2,3) {$e_2$};  
     \node[font=\tiny] at (.2,19) {$e_1$};  
     
     \draw[ color=black, thick, snake=coil, segment aspect=0,%
        segment length=11pt] (1,3)--(7.9,3) ;
          \draw[ color=black, thick, snake=coil, segment aspect=0,%
        segment length=11pt] (1,19)--(7.9,19) ;

     \node[font=\large] at (4,10.8) {$D_i$};

     \draw[draw=black,dashed] (7.5,0)--(7.5,22);
     \draw[draw=black,dashed] (1,-2) rectangle (14,24);

     \draw[draw=black, fill=blue,opacity=.13] (1,0) rectangle (14,22);
      \draw[draw=black] (1,0) rectangle (14,22);
     \draw[color=black,ultra thick] (1,0)--(1,22);
     \draw[color=black,thin,fill=DarkGreen,opacity=.3] (1,9) rectangle (10,13);
     
\draw[ color=black, thick] (1,0) to [bend right=45] (11,11) ;
\draw[ color=black, thick] (11,11) to [bend right=45] (1,22) ;
\fill[ color=white, opacity=.5] (1,0) to [bend right=45] (11,11)--(14,11)--(14,0)--(1,0);
\fill[ color=white, opacity=.5] (11,11) to [bend right=45] (1,22)--(14,22)--(14,11)--(11,11);

     \node[font=\large] at (11,1.2) {$B_{2i-1}$};
     \node[font=\small] at (15,11) {$0$};

    \end{scope}

  \end{tikzpicture}\vspace{-0.2cm}
  \caption{Left: the event $\Gamma_{2i-1}$ where the interface of the component of $\partial_\west B_{2i-1}$ does not intersect $D_i$, and is connected to two edges $e_1,e_2$ that have no bridges. Right: the event $\Gamma_{2i}$ where the interface of the component of $\partial_\west B_{2i}\cup \partial_{\north,\south} D_i$ does not intersect the dashed line. Overall, the intersection $\Gamma_{2i-1}\cap \Gamma_{2i}$ pushes the interface forward by $\ell n$.
  \vspace{-0.2cm}
  } 
  \label{fig:pushing-the-interface-2}
\end{figure}
Thus, by monotonicity, we consider the probability of the event $\mathcal I\cap D_i=\emptyset$ in $\Gamma_{2i-1}$ holding for a sample from $\pi^{\zeta,0}_{B_{2i-1}}$. We claim that for some $c(\delta,q)>0$, 
\begin{align*}
\mathbb E[\pi_{B_{2i-1}}^{\zeta,0}(\mathcal I\cap D_i = \emptyset) \mid F_{i-1}] & \geq \mathbb E_{\pi_{\mathbb Z^2}^1} [\pi_{B_{2i-1}}^{\xi,0}(\mathcal I \cap D_i = \emptyset)] \\
& \gtrsim \pi_{\Lambda_{(1-2\delta)n,3\ell}}^{1,0}(\mathcal I \cap \llb \tfrac n2 -\delta n,\tfrac n2 +\delta n\rrb \times \llb  0,\tfrac {5\ell}2\rrb=\emptyset)-e^{-c \ell}
\end{align*}
where the expectation is over all boundary conditions $\xi$ induced by $\pi_{\mathbb Z^2}^1$ on $\partial_\west B_{2i-1}$ and $(1,0)$ boundary conditions denote wired on $\partial_\south \Lambda_{(1-2\delta)n,3\ell}^{1,0}$. Indeed, the second inequality follows from considering the $\ell$-enlargement $E_{\ell,2i-1}$ of $B_{2i-1}$ which is its concentric rectangle with extra side length $\ell$. If there is a wired circuit in $E_{\ell,2i-1} - B_{2i-1}$ under $\pi_{\mathbb Z^2}^1$ (by~\eqref{eq:exp-decay} this has probability $1-e^{-c\ell}$), we can replace the expectation over b.c.\ induced by $\pi_{\mathbb Z^2}^1$ with an expectation over b.c.\ induced by $\pi_{E_{\ell,2i-1}}^1$. Then extending the free boundary conditions on the other three sides of $B_{2i-1}$ all the way to $\partial_\west E_{\ell,2i-1}$ and rotating yields the second inequality.
By Proposition~\ref{prop:cylinder-midpoint-estimate} with the choices $h=3\ell$ and $\rho=\frac 56$, there exists $c(\delta,q)>0$ so that the probability in the right-hand side above is at least order $e^{-cn^{2\epsilon}}$ (see e.g.,~\cite[\S5 and Fig.\ 7]{GL16a} for a similar monotonicity argument).

Moreover, by the exponential decay of dual-connectivities, it is clear that for any $\zeta \in \mathcal E_{f_n}$, we have $\pi_{B_{2i-1}}^{\zeta,0}(\tilde \Gamma_{2i-1})\geq \eta$ for some $\eta(q)>0$. Thus by the FKG inequality, 
\begin{align*}
 \mathbb E[\pi_{B_{2i-1}}^{\zeta,0}(\Gamma_{2i-1})] \gtrsim e^{-cN^{2\epsilon}}\,.
\end{align*}

Let $\mathcal I_{2i-1}$ be the interface revealed by the component of $\partial _\west B_{2i-1}$. Observe that because $\Gamma_{2i-1}$ is an increasing event, conditioned on $\Gamma_{2i-1}$, if we reveal $\mathcal I_{2i-1}$ from east to west, under the monotone coupling of $\pi^{\zeta,1}_{B_{2i-1}}$ to $\pi^{\zeta,0}_{B_{2i-1}}$ the same edges would also be open under $\pi^{\zeta,1}_{B_{2i-1}}$; the same is also true of the edges that constitute $\tilde \Gamma_{2i-1}$. Having revealed these sets of open edges under both $\pi^{\zeta,1}_{B_{2i-1}}$ and $\pi^{\zeta,0}_{B_{2i-1}}$, by the domain Markov property (there can not be distinct bridges over the interface we have revealed),  $\bar X_{t_{2i-1}}^1\restriction_{D_i} = \bar X_{t_{2i-1}}^0\restriction_{D_i}$ with probability at least $(1-O(e^{-n}))e^{-cn^{2\epsilon}}$.

Now consider the next time increment $[t_{2i-1},t_{2i})$ on $B_{2i}$. Under the above events, the configuration $\bar X_{t_{2i-1}}^0\restriction_{D_i} \succeq \pi_{\mathbb Z^2}^1 (\cdot \restriction_{D_i})$, whence by~\eqref{eq:exp-decay}, with probability at least $1-e^{-c\delta n}$, there is a pair of primal horizontal crossings of the top and bottom halves of $D_i$ connecting  $\partial_{\west}B_{2i-1}$ to $\mathcal I_{2i-1}$. In that case, the distribution on boundary conditions induced by $\bar X^0_{t_{2i-1}}$ (as well as $\bar X^1_{t_{2i-1}}$) on $\partial_{\north, \south} B_{2i} \cap D_i$ dominates wired-at-infinity. Again, since we are working under the event $\Upsilon_{2i}$, we just consider the event in $\Gamma_{2i}$ under $\pi^{\zeta,0}_{B_{2i}}$. By applying~\eqref{eq:exp-decay} and enlarging the domains under consideration as in the earlier bound on $ \Gamma_{2i-1}$, we obtain for some $c(\delta,q)>0$,
\begin{align*}
\mathbb P\Big( \Gamma_{2i} \mid A_{2i-2}^c,  \Gamma_{2i-1}, F_{i-1}\Big) & \geq (1-O(e^{-n}))\mathbb E_{\pi_{\mathbb Z^2}^1} \big[ \pi_{B_{2i}}^{\xi,0,(i+1/2)\ell}(\Gamma_{2i})\big] \\
& \geq (1-O({e^{-n}})) \mathbb E_{\pi_{\mathbb Z^2}^1}\big[ \pi_{\Lambda_{(1-2\delta) n,3\ell /2}}^{0,\xi}(\mathcal I\cap \Lambda_{(1-2\delta) n, \ell}= \emptyset)\big] \\
& \geq (1-O(e^{-n})) (\pi_{\Lambda_{n,2\ell}}^{0,1}(\mathcal I \cap \Lambda_{n,3\ell/2}= \emptyset)-e^{-c\ell})\,.
\end{align*}
Here, the boundary conditions $(\xi,0)$ in the first line denote $\xi$ (over which we take an expectation) induced on $\partial B_{2i}\cap \{(x,y):x=(i+\frac 12) \ell\}$, and free elsewhere on $\partial B_{2i}$, and the boundary conditions in the second and third lines denote free on $\partial_\north$ of the boundary and respectively $\xi$ and wired elsewhere. The second inequality is a simple consequence of monotonicity in boundary conditions and the third inequality follows from enlarging $B_{2i}$ by $\ell/2$ up to an error of $e^{-c \ell}$ coming from~\eqref{eq:exp-decay}.
By Lemma~\ref{lem:subset-surface-tension} with $\phi=0$ and $b=1$, there exists $c(\delta,q)>0$ such that the probability on the right-hand side above is bounded below by $(1-O(e^{-n})) (1-e^{-cn^{2\epsilon}})$. In that case, revealing the interface from east to west, we can couple $\bar X^0_{t_{2i}}$ to $\bar X^1_{t_{2i}}$ beyond the interface (see also Fig.~\ref{fig:pushing-the-interface-2}), so
\begin{align}\label{eq:A-2i-c-bound}
\mathbb P(A_{2i}^c \mid A_{2i-2}, F_{i-1}) \geq (1-O(e^{-n}))(1-e^{-cn^{2\epsilon}})e^{-cn^{2\epsilon}} \gtrsim \exp({-cn^{2\epsilon}})\,.
\end{align}

Finally, we claim that under the intersection of all the above events, with probability $1-O(e^{-c\ell})-O(n^2 e^{-C_q f_n})$, the boundary conditions induced by $\bar X^{1/0}_{t_{2i}}$ on $\partial_\west B_{2i+1}$ and $\partial_\west B_{2i+2}$ are in $\mathcal E_{f_n}$ and dominate wired-at-infinity, which combined with~\eqref{eq:A-2i-c-bound} defines the desired set set $F_i$ such that 
$$\mathbb P(A_{2i}^c, F_{i} \mid A_{2i-2}, F_{i-1}) \gtrsim e^{-cn^{2\epsilon}}\,.$$
 Recall that the configuration on $D_i$ under $ \Gamma_{2i-1}$ and $F_{i-1}$ dominates $\pi_{\mathbb Z^2}^1 \restriction_{D_i}$. Then, with probability $1-2e^{-c\delta n^{\frac 12+\epsilon}}$, $D_i$ contains two horizontal crossings connecting $\partial_\west D_i$ to $\mathcal I_{2i-1}$; since we are also conditioning on $ \Gamma_{2i}$, averaging over configurations on $D_i$, with probability $1-3e^{-c\delta n^{\frac 12+\epsilon}}$, $\partial_\west B_{2i+2}$ is surrounded by a wired circuit in $\bar X^0_{t_{2i}}$. Similarly, under $\bar X^0_{t_{2i}}$, conditional on $ \Gamma_{2i-1}$, $ \Gamma_{2i}$ and $F_{i-1}$, the configuration on 
\[
D_i'= \llb i\ell, (i+2)\ell \rrb \times \llb 0,\delta n\rrb \cup \llb (1-\delta) n,n\rrb
\]
below the interface revealed by $ \Gamma_{2i}$ dominates $\pi_{\mathbb Z^2}^1$ so that with probability $1-2e^{-c\delta n^{\frac 12+\epsilon}}$, there are horizontal primal connections to that interface in both halves of $D_i'$. In that case, averaging over configurations in $D_i'$ with probability $1-3e^{-c\delta  n^{\frac 12+\epsilon}}$, there is a wired circuit around $\partial_\west B_{2i+1}$ in $\bar X^0_{t_{2i}}$ so that the distribution over boundary conditions induced on $\partial_\west B_{2i+1}$ also dominates wired-at-infinity. Moreover, as seen in the proof of Proposition~\ref{prop:canonical-paths}, a boundary condition dominating wired-at-infinity is in $\mathcal E_{f_n}$ with probability $1-O(n^2 e^{-C_q f_n})$. A union bound over the above concludes the proof. \end{proof}

As a result, by item (2) of  Claim~\ref{claim:induction2}, there exists some $c(\delta,q)>0$ for which
\begin{align*}
\mathbb P(A_{2N}^c) \gtrsim \exp\big(-cN n^{2\epsilon}\big) \gtrsim \exp\big(-cn^{ \frac12 +\epsilon}\big)\,.
\end{align*}
Moreover, on that event, with high probability, the boundary conditions on $\partial_\west B_{2N+1}$ induced by both $\bar X^{1}_{2N}$ and $\bar X^0_{2N}$ dominate wired-at-infinity. On the event $\Upsilon_{2N+1}$, with probability $1-O(e^{-n})$ one can couple $\bar X_{t_{2N+1}}^1$ and $\bar X_{t_{2N+1}}^0$ to agree on $B_{2N+1}$, leading the two chains to be coupled on all of $\Lambda$. 
(It is only at this final step where there is a difference between the $(p,1,0)$ and $(p,1)$ boundary conditions; clearly, if the coupling on $B_{2N+1}$ succeeds in the former situation, it also succeeds in the latter.) 
\end{proof}

\section{Slow mixing with phase-symmetric boundary conditions}\label{sec:slow-mixing}
For a reversible chain with transition kernel $P(x,y)$ and stationary distribution $\pi$, define the edge measure $Q$ between $A,B\subset\Omega$ and conductance of the chain, $\Phi$, by
\[Q(A,B)=\sum _{\omega\in A} \pi(\omega)\sum_{\omega'\in B}P(\omega,\omega')\,,\qquad \mbox{and}\qquad \Phi=\max_{\mathcal A\subset \Omega}\frac {Q(\mathcal A,\mathcal A^c)}{\pi(\mathcal A)\pi(\mathcal A^c)}\,.
\]
The Cheeger inequality relates these to the gap  (see, e.g.,~\cite[\S7]{LPW17}), by stating that 
\begin{align}\label{eq:Cheeger-constant}
2\Phi\geq \gap \geq  {\Phi^2}/{2}\,.
\end{align}

\subsection*{The torus}
In Theorem~2 of \cite{GL16a}, the authors used the above to construct an exponential bottleneck relying heavily on the topology of the torus and the exponential decay of correlations under $\pi^{0}_{\mathbb Z^2}$ at a discontinuous phase transition point. We restate the result for the critical Swendsen--Wang dynamics for all $q>4$, which follows from the sharp identification in \cite{DGHMT16} of the discontinuity of the phase transition for all $q>4$.  

\begin{theorem}[{\cite[Theorem 3]{GL16a}}, given the result of \cite{DGHMT16}]\label{thm:torus}
Let $q>4$, and consider the Swendsen--Wang dynamics on $(\mathbb Z/n \mathbb Z)^2$ at $\beta= \beta_c(q)$. There exists $c(q)>0$ such that
\[\tmix\gtrsim \exp(cn)\,.
\]
\end{theorem}

\noindent Observe that exploiting the topology of the torus, unlike the other results in this paper, the above requires neither validity of the cluster expansion nor positivity of surface tension, and therefore holds up through $q>4$. On the other hand, for $q$ that is sufficiently large, slow mixing at $\beta=\beta_c$ was previously shown in~\cite{BCT12} (in any dimension).

\subsection*{Order-disorder mixed boundary conditions}
Although the proof of slow mixing on the torus at a discontinuous phase transition relies heavily on the topology of the torus (see proof of Theorem 2 in \cite{GL16a}) we can---at least for sufficiently large $q$---use a similar approach to prove slow mixing in the presence of mixed wired-free boundary conditions. Exploiting the self-duality, we see that such boundary conditions still exhibit an exponential bottleneck, slowing down the Swendsen--Wang dynamics.

\begin{definition}\label{def:alt-bc}
Let $a_n,b_n,c_n,d_n \in \partial \Lambda_{n,n}$ be a set of marked vertices ordered clockwise from the origin around $\partial \Lambda_{n,n}$ (by rotational symmetry, without loss of generality assume $a_n\in \partial_\west \Lambda_{n,n}$).
The \emph{mixed boundary conditions} on $(a_n,b_n,c_n,d_n)$ are those that are red on the clockwise boundary arcs $(a_n,b_n)$ and $(c_n,d_n)$ and  free on $(b_n,c_n)$ and $(d_n,a_n)$---all connected subsets of $\partial \Lambda_{n,n}$. We say that $(a_n,b_n,c_n,d_n)$ are \emph{$\epsilon$-separated} if $a_n\in \partial_\west \Lambda_{n,n}$, at least one of $\{b_n,c_n,d_n\}$ is not contained in $\llb 0,\epsilon n\rrb \times \llb 0,n\rrb$, and
\[\min _{i,j\in a,b,c,d; \, i\neq j} \|i_n - j_n\|_\infty\geq \epsilon n\,.
\]
\end{definition}

\begin{remark}\label{rem:epsilon-separated}
The requirement of $\epsilon$-separation in our consideration of mixed boundary conditions arises from the fact that if $(a_n,b_n,c_n,d_n)$ were all on $\partial_\west \Lambda_{n,n}$ repeating the proof of Theorem~\ref{mainthm:2} with such boundary conditions would yield that the mixing time is in fact sub-exponential. Clearly, if the four marked vertices are sufficiently close to being on one side or to each other, a similar picture would emerge. The requirement of macroscopic separation ensures that the bottleneck is exponential in $n$.
\end{remark}

With Definition~\ref{def:alt-bc} in hand, to prove Theorem~\ref{mainthm:1}, by rotational symmetry, we wish to prove the following: let $q$ be large, $\epsilon>0$, and consider the Swendsen--Wang dynamics for the critical Potts model on $\Lambda=\Lambda_{n,n}$ with mixed boundary on $(a_n,b_n,c_n,d_n)$ that are $\epsilon$-separated.
Then there exists $c(\epsilon, q)>0$ such that
\[\tmix\gtrsim \exp(c n)\,.
\]

\begin{proof}[\emph{\textbf{Proof of Theorem \ref{mainthm:1}}}]
By \eqref{eq-ullrich2} and Fact~\ref{fact:one-color}, it suffices to prove the bound for the FK Glauber dynamics with mixed FK boundary conditions that are wired on the boundary arcs $(a_n,b_n)$ and $(c_n,d_n)\subset \partial\Lambda_{n,n}$ (and the two boundary arcs are wired together) and free elsewhere (denoted by $\pi_\Lambda^{\mathrm{mixed}}$). Observe that by planarity,
\[\big\{(a_n,b_n)\longleftrightarrow (c_n,d_n)\big\} =\big\{(b_n,c_n)\stackrel{\ast}\longleftrightarrow (d_n,a_n)\big\}^c\,,
\]
and therefore, either
\[\pi^{\alt}_{\Lambda}\big((a_n,b_n)\longleftrightarrow (c_n,d_n)\big)\leq  \frac 12\,, \quad \mbox{ or }\quad
\pi^{\alt}_{\Lambda}\big((b_n,c_n)\stackrel{\ast}\longleftrightarrow (d_n,a_n)\big)\leq\frac 12\,.
\]
By self-duality of the class of $\epsilon$-separated, mixed boundary conditions, we can suppose without loss of generality that we are in the former case.

Recall the definition of the strips $\mathcal S_{b,h,\phi}, \mathcal H^{\pm}_{b,\phi}$ in Definition~\ref{def:strips}, and let $a_n=(a_n^1,a_n^2)$, and likewise for $b_n,c_n,d_n$.
Then let $\phi_{a,d}= \tan^{-1} (\frac {d^2_n-a^2_n}{d^1_n-a^1_n})$ and $\phi_{b,c}= \tan^{-1}(\frac {c_n^2 - b_n^2}{c_n^1 - b_n^1})$. Observe that by the $\epsilon$-separation of $(a_n,b_n,c_n,d_n)$, one of $\phi_{a,d}$ and $\phi_{b,c}$ is in $[-\frac \pi 2+\delta,\frac \pi 2 - \delta]$ for some small enough $\delta>0$ depending only on $\epsilon$. Suppose without loss of generality that for some $\delta(\epsilon)>0$, $\phi_{a,d}\in [-\frac \pi 2 +\delta, \frac \pi 2-\delta]$ and consider the strip
\begin{align*}
S & = \cH^+_{a^1_n, \phi_{a,d}}\cap \cH^-_{a^1_n+\epsilon^2 n, \phi_{a,d}}\cap \Lambda\,,
\end{align*}
Geometrically, by definition of $\epsilon$-separation, $S$ satisfies $S \cap \partial \Lambda \subset (a_n,b_n)\cup (c_n,d_n)$. There is some $x,h,\phi=\phi_{a,d}$ such that $S=\cS_{x,h,\phi_{a,d}}\cap \Lambda$; fix that $x\in \mathbb R_+$, $h= \epsilon^2/2$. Define $\partial_{\north} S=S \cap  \mathcal H^+_{x+h-1,\phi}$ and $\partial_\south S= S \cap \mathcal H^-_{x-h+1,\phi}$, and let
\[\mathcal A=\left\{(a_n,b_n)\stackrel{ S}\longleftrightarrow (c_n,d_n)\right\}\,
\]
be the bottleneck set whose conductance $Q(\mathcal A,\mathcal A^c)/(\pi(\mathcal A)\pi (\mathcal A^c))$ we bound. Since
\[\mathcal A \subset \left\{(a_n,b_n)\longleftrightarrow (c_n,d_n)\right\}\,,
\]
we have that $\pi_{\Lambda}^{\alt}(\mathcal A^c)>\frac 12$. Therefore, we can write 
\[\gap\leq 2\Phi \leq \frac {2Q(\mathcal A,\mathcal A^c)}{\pi_\Lambda^{\alt}(\mathcal A)\pi_{\Lambda}^{\alt}(\mathcal A^c)}\leq  4\pi_\Lambda^{\alt}(\partial \mathcal A\mid \mathcal A)\,,
\]
(where $\partial \mathcal A:=\{\omega:P(\omega,\mathcal A^c)>0\}$ and we used a worst-case bound of $1$ on the transition rates in $Q(\mathcal A, \mathcal A^c)$), in which case it suffices to prove that for some $c(q)>0$,
\[\pi_{\Lambda}^{\alt}(\partial \mathcal A\mid \mathcal A)\lesssim \exp(-c\epsilon^4 n)\,.
\]

For $\omega\in \mathcal A$, in order for $P(\omega,\mathcal A^c)$ to be positive ($\omega\in\partial \mathcal A$), there must exist an edge $e$ in $\mathcal S$ that is \emph{pivotal} to $\mathcal A$, i.e., $\omega(e)=1$ and $\omega'=\omega-\{e\}\notin \mathcal A$. We estimate the probability $\pi^{\alt}_\Lambda(\partial \mathcal A\mid \mathcal A)$ by taking a union bound over the probability of any edge, $e$, in $E(S)$ being pivotal to $S$.

First examine whether $e$ is closer in its $y$ coordinate to $\partial_\north S$ or $\partial_\south S$. Suppose without loss of generality, we are in the former case, whence we expose the north-most primal crossing of $S$, under $\pi_\Lambda^{\alt}(\omega \mid \mathcal A)$ (revealing, first, the configuration on $\Lambda\cap \mathcal H^+_{x+h-1,\phi}$, then the dual-components of $\partial_\north S$ in $S$). Such a crossing exists by conditioning on $\mathcal A$.

Denote by $\zeta$ the  horizontal crossing we have revealed as such. By the conditioning on $S$, it is clear that $\zeta$ must connect $(a_n,b_n)$ to $(c_n,d_n)$ in $S$. In order for $e$ to be pivotal to $S$, $e$ must be an open edge in $\zeta$ and there must exist a dual crossing connecting $e$ to $\partial _\south S$. Let $D$ be the southern connected component of $E(\Lambda)-\zeta$; we wish to bound
\[\pi_{\Lambda}^{\alt}\left(e\stackrel{D^\ast}\longleftrightarrow \partial_\south S \mid \mathcal A, \zeta,\omega\restriction_{\zeta}=1\right)\,.
\]

By monotonicity in boundary conditions, if we let $R= D\cup S$, for every such $\zeta$, 
\[\pi_{\Lambda}^{\alt}(\omega\restriction_{D}\mid \mathcal A,\zeta,\omega\restriction_{\zeta}=1) = \pi_{D}^{1,0}\succeq \pi_R^{1,0}(\omega\restriction_{D})\,,
\]
where $(1,0)$ boundary conditions denote free on $ \partial R - S$ and wired elsewhere.

We can decompose the probability
\[\pi_R^{1,0}\left(e\stackrel{D^\ast}\longleftrightarrow \partial_{\south} S\right)\leq \pi_R^{1,0}\left(e\stackrel{\ast}\longleftrightarrow \partial_{\south} S \right)
\]
into the event $\Gamma_1$ that the dual-component of $\partial_\south S$ (and thus the interface of $R$ with $(1,0)$ boundary conditions) is a subset of $S \cap \mathcal H^-_{x-h/2,\phi}$, and $\Gamma_1^c$. Under $\Gamma_1^c$, since $e$ is closer in its $y$-coordinate to $\partial_\north S$, the vertical distance between $e$ and $\partial_\south S$ is at least $h/2$ so that $e\notin S\cap \cH^-_{x-h/2,\phi}$ and $e$ cannot be dual-connected to $\partial_\south S$.

Bounding the probability of $\Gamma_1$ by monotonicity in boundary conditions and Proposition~\ref{prop:strip-surface-tension}, there exists $c(q)>0$ such that for every $\zeta$, 
\[\pi_{\Lambda}^{\alt}(\Gamma_1\mid \mathcal A,\zeta, \omega\restriction_{\zeta}=1)\leq \pi_{\mathcal S_{a_n^1,h/2,\phi}\cap \Lambda}^{1,0,a_n^1,\phi}(\mathcal I \not \subset \mathcal S_{a_n^1,h/4,\phi}\cap \Lambda)\lesssim n^2 \exp(-ch^2 n)\,.
\]
Under $\Gamma_1$ we can take a worst case bound of one on the probability of $e\stackrel{D^{\ast}}\longleftrightarrow \partial_\south S$. Therefore, for some $c(q)>0$, we have $\pi_{\Lambda}^{\alt} (\partial \mathcal A\mid \mathcal A)\lesssim \exp(-c\epsilon^4 n)$.

Using the above as a bound on $Q(\mathcal A,\mathcal A^c)/\pi_\Lambda^{\alt}(\mathcal A)$ in Eq.~\eqref{eq:Cheeger-constant} and plugging into Eq.~\eqref{eq:Cheeger-constant} implies for the FK Glauber dynamics and, by Eq.~\eqref{eq-ullrich1}, Swendsen--Wang dynamics with mixed order-disorder boundary conditions,  $\gap^{-1}\gtrsim \exp(c\epsilon^4 n)$. \end{proof}

\subsection*{Acknowledgment} We thank the anonymous referee for many useful suggestions and comments. R.G.\ was supported in part by NSF grant DMS-1507019. E.L.\ was supported in part by NSF grant DMS-1513403.

\bibliographystyle{abbrv}
\bibliography{boundary-conditions-mixing-v1-with-pictures}

\end{document}